\documentclass[11pt,a4paper]{article}

\usepackage{graphicx,latexsym,euscript,makeidx,color,bm}
\usepackage{amsmath,amsfonts,amssymb,amsthm,thmtools,mathtools,mathrsfs,enumerate}
\usepackage[colorlinks,linkcolor=blue,citecolor=red]{hyperref}

\usepackage{geometry}
\allowdisplaybreaks[4]
  \def\hA{\widehat{A}} \def\cA{{\cal A}}  \def\bu{\bar{u}}
 \def\sB{\mathscr{B}} \def\hB{\widehat{B}}   \def\hv{\hat{v}}
  \def\hC{\widehat{C}} \def\cC{{\cal C}}  \def\bA{\bar{A}}
  \def\hD{\widehat{D}}   \def\bB{\bar{B}}
\def\dbE{\mathbb{E}}   \def\cE{{\cal E}}  \def\bC{\bar{C}}
\def\dbF{\mathbb{F}} \def\sF{\mathscr{F}}    \def\bD{\bar{D}}
     \def\bQ{\bar{Q}}
\def\dbH{\mathbb{H}}   \def\cH{{\cal H}}  \def\bS{\bar{S}}
     \def\bR{\bar{R}}
     \def\bX{\bar{X}}
   
   \def\cL{{\cal L}}

\def\dbP{\mathbb{P}}   \def\cP{{\cal P}}
  \def\hQ{\widehat{Q}} \def\cQ{{\cal Q}}
\def\dbR{\mathbb{R}}  \def\hR{\widehat{R}} \def\cR{{\cal R}}
\def\dbS{\mathbb{S}}  \def\hS{\widehat{S}} \def\cS{{\cal S}}
   
 \def\sU{\mathscr{U}}  \def\cU{{\cal U}}

   \def\cX{{\cal X}}
   \def\cY{{\cal Y}}
   \def\cZ{{\cal Z}}
\def\nid{\,|\,}       \def\lt{\left}          \def\hb{\hbox}
\def\ms{\medskip}     \def\rt{\right}         \def\ae{\text{a.e.}}
\def\h{\hat}          \def\lan{\langle}       \def\as{\text{a.s.}}
\def\q{\quad}         \def\ran{\rangle}       \def\tr{\hb{tr$\,$}}
             \def\diag{\hb{diag\,}}
\def\no{\noindent}          
\def\hp{\hphantom}         \def\scp{\scriptscriptstyle}
\def\nn{\nonumber}         \def\scT{\scp T}
\def\rf{\eqref}       \def\Blan{\Big\lan\!\!} \def\scN{\scp N}
\def\cd{\cdot}        \def\Bran{\!\!\Big\ran} 
\def\deq{\triangleq}  \def\({\Big(}           \def\les{\leqslant}
       \def\){\Big)}           \def\ges{\geqslant}
   \def\[{\Big[}           
\def\wh{\widehat}     \def\]{\Big]}           

\def\a{\alpha}       \def\l{\lambda}    \def\D{\varDelta}
\def\b{\beta}        \def\t{\tau}       \def\F{\varPhi}
\def\d{\delta}           \def\G{\varGamma}
\def\e{\varepsilon}      \def\L{\varLambda}
\def\f{\varphi}      \def\p{\phi}       \def\Om{\varOmega}
       \def\si{\sigma}    \def\Si{\varSigma}
\def\i{\infty}             \def\Th{\varTheta}
       \def\vP{\varPi}    

\newtheoremstyle{thry}
{}      
{}      
{\sl}   
{}      
{\bf}   
{.}     
{.5em}  
{}      
\theoremstyle{thry}

\newtheorem{theorem}{Theorem}[section]
\newtheorem{proposition}[theorem]{Proposition}
\newtheorem{corollary}[theorem]{Corollary}
\newtheorem{lemma}[theorem]{Lemma}

\theoremstyle{definition}
\newtheorem{definition}[theorem]{Definition}

\theoremstyle{remark}
\newtheorem{remark}[theorem]{Remark}

\def\punct{}
\newtheoremstyle{dotless}{}{}{\rm}{}{\bf}{\punct}{.5em}{}
\theoremstyle{dotless}
\newtheorem{taggedassumptionx}{}
\newenvironment{taggedassumption}[1]
 {\renewcommand\thetaggedassumptionx{#1}\taggedassumptionx}
 {\endtaggedassumptionx}

\makeatletter
   
   \@addtoreset{equation}{section}
   \newcommand{\setword}[2]{%
   \phantomsection
   #1\def\@currentlabel{\unexpanded{#1}}\label{#2}%
   }
\makeatother

\begin{document}

\title{\bf Periodic Exponential Turnpike Phenomenon in Mean-Field Stochastic Linear-Quadratic Optimal Control}
\author{Jingrui Sun\thanks{Department of Mathematics and SUSTech International Center for Mathematics,
                           Southern University of Science and Technology, Shenzhen, Guangdong,
                           518055, China (Email: {\tt sunjr@sustech.edu.cn}).
                           This author is supported in part by NSFC grants 12322118 and 12271242, and
                           Shenzhen Fundamental Research General Program JCYJ20220530112814032. }
~~~~~
Lvning Yuan\thanks{Corresponding author. Department of Mathematics,
		           Southern University of Science and Technology, Shenzhen, Guangdong,
		           518055, China (Email: {\tt yuanln@sustech.edu.cn}).}
~~~~~
Jiaqi Zhang\thanks{Department of Mathematics, Southern University of Science and Technology,
                   Shenzhen, Guangdong, 518055, China (Email: {\tt zhangjq3@sustech.edu.cn}). }
}

\maketitle

\no{\bf Abstract.}
The paper establishes the exponential turnpike property for a class of mean-field stochastic
linear-quadratic (LQ) optimal control problems with periodic coefficients.
It first introduces the concepts of stability, stabilizability, and detectability for stochastic
linear systems.
Then, the long-term behavior of the associated Riccati equations is analyzed under stabilizability
and detectability conditions.
Subsequently, a periodic mean-field stochastic LQ problem is formulated and solved.
Finally, a linear transformation of the periodic extension of its optimal pair
is shown to be the turnpike limit of the initial optimal control problem.

\ms
\no{\bf Key words.}
Exponential turnpike property, stochastic optimal control, periodic, mean-field, linear-quadratic,
Riccati equation, stabilizability, detectability.

\ms
\no{\bf AMS 2020 Mathematics Subject Classification.}  49N10, 49N20, 49N80, 93E20.

\section{Introduction}\label{Sec:Intro}

Let $(\Om,\sF,\dbP)$ be a complete probability space on which a standard one-dimensional
Brownian motion $W=\{W(t),\sF_t;\,t\ges 0\}$ is defined, where $\dbF=\{\sF_t\}_{t\ges0}$
is a filtration satisfying the usual conditions. Consider the following controlled linear
mean-field stochastic differential equation (SDE, for short)
\begin{equation}\label{TP:state}\left\{\begin{aligned}
dX(t) &= \big\{A(t)X(t) + \bA(t)\dbE[X(t)] + B(t)u(t) + \bB(t)\dbE[u(t)] + b(t) \big\}dt \\
      &\hp{=\ } +\big\{C(t)X(t) + \bC(t)\dbE[X(t)] + D(t)u(t) + \bD(t)\dbE[u(t)] + \si(t)\big\}dW(t), \\
 X(0) &= x
\end{aligned}\right.\end{equation}
and the quadratic cost functional
\begin{align}\label{TP:cost}
J_{\scT}(x;u(\cd))
&\deq \dbE\int_0^T\Bigg[\Blan\begin{pmatrix*}[l]Q(t) & \!S(t)^\top \\ S(t) & \!R(t)\end{pmatrix*}\!
                          \begin{pmatrix}X(t) \\ u(t)\end{pmatrix}\!,
                          \begin{pmatrix}X(t) \\ u(t)\end{pmatrix}\Bran
+2\Blan\begin{pmatrix}q(t) \\ r(t)\end{pmatrix}\!,\begin{pmatrix}X(t) \\ u(t)\end{pmatrix}\Bran\nn\\
&\hp{=\ } +\Blan\begin{pmatrix*}[l]\bQ(t) & \!\bS(t)^\top \\ \bS(t) & \!\bR(t)\end{pmatrix*}\!
                \begin{pmatrix}\dbE[X(t)] \\ \dbE[u(t)]\end{pmatrix}\!,
                \begin{pmatrix}\dbE[X(t)] \\ \dbE[u(t)]\end{pmatrix}\Bran\Bigg] dt,
\end{align}
where the coefficients $A(\cd)$, $\bA(\cd)$, $B(\cd)$, $\bB(\cd)$, $b(\cd)$, $C(\cd)$, $\bC(\cd)$,
$D(\cd)$, $\bD(\cd)$, $\si(\cd)$ and $Q(\cd)$, $\bQ(\cd)$, $S(\cd)$, $\bS(\cd)$, $R(\cd)$, $\bR(\cd)$,
$q(\cd)$, $r(\cd)$ are suitable dimensional deterministic matrix/vector-valued functions defined on $[0,\i)$,
with $Q(\cd),\bQ(\cd),R(\cd)$ and $\bR(\cd)$ being symmetric matrices.
The superscript $\top$ denotes the transpose of matrices, and $\lan\cd\,,\cd\ran$ denotes the Frobenius
inner product of two matrices.
All of above coefficients are measurable periodic functions with a common period $\t>0$,
that is, for $\f(\cd)$ being one of the above functions,
\begin{align*}
 \f(t+\t) = \f(t), \q\forall t\ges 0.
\end{align*}
The vector $x\in\dbR^n$ in \rf{TP:state} is called an {\it initial state}, and the process $u(\cd)$,
called a {\it control}, is selected from the following space:
\begin{align*}
\sU[0,T] \deq \lt\{u:[0,T]\times\Om\to\dbR^m \bigm| u(\cd)\in\dbF\text{ and }\dbE\int_0^T|u(t)|^2dt<\i \rt\},
\end{align*}
where $u(\cd)\in\dbF$ means that $u(\cd)$ is progressively measurable with respect to the filtration $\dbF$.
For a fixed time horizon $T>0$, the {\it mean-field stochastic linear-quadratic (LQ, for short) optimal
control problem} can be stated as follows.

\ms

{\bf Problem (MFLQ)$_{\scT}$.}
For a given initial state $x\in\dbR^n$, find a control $\bu_{\scT}(\cd)\in\sU[0,T]$ such that
\begin{align}\label{PTP:opt-u}
J_{\scT}(x;\bu_{\scT}(\cd)) = \inf_{u(\cd)\in\sU[0,T]} J_{\scT}(x;u(\cd)) \equiv V_{\scT}(x).
\end{align}

\smallskip

The process $\bu_{\scT}(\cd)$ in \rf{PTP:opt-u} (if exists) is called an {\it optimal control}
of Problem (MFLQ)$_{\scT}$ for the initial state $x$,
the corresponding state process $\bX_{\scT}(\cd)$ is called an {\it optimal state process},
the pair $(\bX_{\scT}(\cd),\bu_{\scT}(\cd))$ is called an {\it optimal pair},
and the function $V_{\scT}(\cd)$ is called the {\it value function} of Problem (MFLQ)$_{\scT}$.

\ms

The above Problem (MFLQ)$_{\scT}$ was initially investigated by Yong \cite{Yong2013},
in which the nonhomogeneous terms $b(\cd)$, $\si(\cd)$, $q(\cd)$, and $r(\cd)$ are absent,
and the weighting coefficients are assumed to be positive definite.
Later, Huang--Li--Yong \cite{Huang-Li-Yong2015} and Ni--Elliott--Li \cite{Ni-Elliott-Li2015}
extended the results of \cite{Yong2013} to the infinite time horizon, in the context of
continuous-time and discrete-time systems.
Sun \cite{Sun2017} further carried out a study on the indefinite version of Problem (MFLQ)$_{\scT}$
and established its open-loop solvability under the uniform-convexity condition.
Since then, a number of interesting works on this topic appeared, including but not limited to  \cite{Ni-Li-Zhang2015,Basei-Pham2019,Li-Li-Yu2020,Lu2020,Wang-Zhang2021,Yang-Tang-Meng2022}.

\ms

Different from the above-mentioned literature primarily focused on a fixed time horizon $T$,
this paper delves into the investigation of the long-time behavior of the optimal pair as $T$
tends to infinity, with particular emphasis on the phenomenon called the {\it exponentially
turnpike property}.

\ms

The turnpike property, originally discussed by Ramsey \cite{Ramsey1928} and von Neumann \cite{Neumann1945},
and subsequently named by Dorfman--Samuelson--Solow \cite{Dorfman-Samuelson-Solow1958}, characterizes
the tendency of optimal trajectories to spend a significant amount of time in the vicinity of a particular
steady state, independent of the time horizon.
Extensive progress has been made in the study of turnpike properties for deterministic optimal
control problems, encompassing diverse perspectives such as finite and infinite-dimensional problems,
as well as discrete-time and continuous-time systems, with noteworthy contributions; see, for example, \cite{Carlson-Haurie-Leizarowitz1991,Porretta-Zuazua2013,Damm-Grune-Stieler-Worthmann2014,Trelat-Zuazua2015,
Zuazua2017,Grune-Guglielmi2018,Zaslavski2019,Lou-Wang2019,Breiten-Pfeiffer2020,Sakamoto-Zuazua2021,
Faulwasser-Grune2022}, and references cited therein.
In stochastic cases, Sun--Wang--Yong \cite{Sun-Wang-Yong2022} established a weak exponential turnpike
property for an LQ optimal control problem with constant coefficients, followed by a deeper study
\cite{Sun-Yong2024:SICON} on mean-field LQ problems.
For the case of periodic coefficients, Sun--Yong \cite{Sun-Yong2024:JDE} explored the turnpike property
for the stochastic LQ problem without mean-field terms.

\ms

Compared to the aforementioned studies, this paper represents a significant extension by exploring
the turnpike property in mean-field stochastic LQ optimal control problems with periodic coefficients.
The key contributions and novelties of our work can be summarized as follows.

\ms

$\bullet$
Both the state equation and the cost functional incorporate expectations of the state and control
processes. This framework exhibits greater generality and holds significant potential for diverse
applications.

\ms

$\bullet$
The exponential turnpike property is established under the exact detectability condition for
stochastic linear systems, which is weaker than the positive definiteness condition assumed
in \cite{Sun-Wang-Yong2022, Sun-Yong2024:SICON, Sun-Yong2024:JDE}.

\ms

$\bullet$
A periodic mean-field stochastic LQ optimal control problem is formulated and solved,
with a linear transformation of its optimal pair's periodic extension 
demonstrated to represent the turnpike limit for Problem (MFLQ)$_{\scT}$.

\ms

$\bullet$
The exponential turnpike property is established not only in terms of trajectory
but also in the distributional sense.
This provides broader potential applications in practical settings.

\ms

The remainder of the paper is organized as follows.
\autoref{Sec:Pre} introduces notation, assumptions, and preliminary results.
\autoref{Sect:Stable-detectable} investigates the stability, stabilizability,
and detectability for stochastic linear systems.
\autoref{Sec:SRE} analyzes the long-term behavior of the associated Riccati equations.
\autoref{Sec:PMFLQ} formulates and solves a periodic mean-field stochastic LQ optimal control problem,
and \autoref{Sec:TP} establishes the exponential turnpike property for Problem (MFLQ)$_{\scT}$.

\section{Preliminaries}\label{Sec:Pre}

In this paper, a vector always refers to a column vector unless otherwise specified.
Let $\dbR^{m\times n}$ be the Euclidean space of all $m\times n$ real matrices,
equipped with the Frobenius inner product
\begin{equation*}
	\lan M, N\ran \deq \tr(M^{\top} N),\q\forall M,N\in\dbR^{m\times n},
\end{equation*}
where $\tr(M^\top N)$ stands for the trace of $M^{\top} N$.
The norm of a matrix $M$ induced by the Frobenius inner product is denoted as $|M|$.
Let $\dbS^n$ be the space of all $n\times n$ symmetric real matrices and $\dbS_+^n$
(respectively, $\bar{\dbS}_+^n$) the space of all $n\times n$ positive definite
(respectively, semi-definite) real matrices.
For $M,N\in\dbS^n$, we write $M\ges N$ (respectively, $M>N$) if $M-N$ is positive
semi-definite (respectively, positive definite).
Denote by $I_n$ the identity matrix of size $n$.
We call a function $F:[0,\i)\mapsto \dbS^{n}$ {\it uniformly positive definite}
if for some constant $\d>0$,
$$ F(t)\ges\d I_n,  \q\ae~t\in[0,\i). $$
For a metric space $\dbH$, let
\begin{align*}
&L^{\i}(0,\i;\dbH)\deq\lt\{\f:[0,\i)\mapsto\dbH\bigm|\f\hb{ is Lebesgue essentially bounded}\rt\},\\
&C([0,\i);\dbH)\deq\lt\{\f:[0,\i)\mapsto\dbH\bigm|\f\hb{ is continuous}\rt\},\\
&C([0,T];\dbH)\deq\lt\{\f:[0,T]\mapsto \dbH\bigm|\f\hb{ is continuous}\rt\}.
\end{align*}
For the sake of brevity, we introduce the following notation:
$$\begin{aligned}
 \hA(\cd)&\deq A(\cd)+\bA(\cd),
&\hB(\cd)&\deq B(\cd)+\bB(\cd),
&\hC(\cd)&\deq C(\cd)+\bC(\cd),
&\hD(\cd) \deq D(\cd)+\bD(\cd), \\
 \hQ(\cd)&\deq Q(\cd)+\bQ(\cd),
&\hS(\cd)&\deq S(\cd)+\bS(\cd),
&\hR(\cd)&\deq R(\cd)+\bR(\cd).
\end{aligned}$$
Throughout the paper, we adopt the following basic assumptions.
\begin{taggedassumption}{(A1)}\label{A1}
The coefficients in \rf{TP:state} and \rf{TP:cost} are periodic functions
with a common period $\t>0$ and satisfy the following boundedness condition:
$$\lt\{\begin{aligned}
&  A(\cd),\bA(\cd),C(\cd),\bC(\cd)\in L^\i(0,\i;\dbR^{n\times n}),
\q B(\cd),\bB(\cd),D(\cd),\bD(\cd)\in L^\i(0,\i;\dbR^{n\times m}), \\
&  Q(\cd),\bQ(\cd)\in L^\i(0,\i;\dbS^n),
\q S(\cd),\bS(\cd)\in L^\i(0,\i;\dbR^{m\times n}),
\q R(\cd),\bR(\cd)\in L^\i(0,\i;\dbS^m), \\
&  b(\cd),\si(\cd),q(\cd)\in L^\i(0,\i;\dbR^n),
\q r(\cd)\in L^\i(0,\i;\dbR^m).
\end{aligned}\rt.$$
\end{taggedassumption}
\begin{taggedassumption}{(A2)}\label{A2}
The functions $R(\cd)$ and $\hR(\cd)$ are uniformly positive definite, and
$$ Q(\cd)-S(\cd)^{\top}R(\cd)^{-1}S(\cd) \ges 0,
\q \hQ(\cd)-\hS(\cd)^{\top}\hR(\cd)^{-1}\hS(\cd) \ges 0. $$
\end{taggedassumption}

The following result, found in \cite{Sun2017}, establishes the unique solvability
of Problem (MFLQ)$_{\scT}$ for a fixed time horizon $T$.

\begin{lemma}\label{lmm:Sun2017}
Let {\rm\ref{A1}--\ref{A2}} hold. Then the system of differential Riccati equations
\begin{equation}\label{DRE:P+vP}\left\{\begin{aligned}
& \dot P_{\scT} + P_{\scT}A + A^\top P_{\scT} + C^\top P_{\scT}C + Q \\
&\hp{P_{\scT}} -\(P_{\scT}B+C^\top P_{\scT}D+S^\top\)\(R+D^\top P_{\scT}D\)^{-1}
\(B^\top P_{\scT}+D^\top P_{\scT}C+S\)=0, \\
&\dot\vP_{\scT} + \vP_{\scT}\hA + \hA^\top\vP_{\scT} + \hQ + \hC^\top P_{\scT}\hC \\
&\hp{\vP_{\scT}} -\(\vP_{\scT}\hB+\hC^\top P_{\scT}\hD+\hS^\top\)\(\hR+\hD^\top P_{\scT}\hD\)^{-1}
\(\hB^\top\vP_{\scT}+\hD^\top P_{\scT}\hC +\hS\)=0, \\
&P_{\scT}(T) =0, \q \vP_{\scT}(T)=0
\end{aligned}\right.\end{equation}
admits a unique solution pair $(P_{\scT}(\cd),\vP_{\scT}(\cd))\in C([0,T];\dbS^n)\times C([0,T];\dbS^n)$
satisfying
$$ P_{\scT}(t)\ges 0, \q \vP_{\scT}(t) \ges 0, \q\forall t\in[0,T]. $$
Furthermore, for each initial state $x\in\dbR^n$, Problem (MFLQ)$_{\scT}$
admits a unique optimal control $\bu_{\scT}(\cd)$, given by
\begin{align}\label{MFLQ:u-rep}
\bu_{\scT}(t) = \Th_{\scT}(t)\(\bX_{\scT}(t)-\dbE[\bX_{\scT}(t)]\)
		+ \wh\Th_{\scT}(t)\dbE[\bX_{\scT}(t)]+\phi_{\scT}(t),
\end{align}
where $\Th_{\scT}(\cd)$, $\wh\Th_{\scT}(\cd)$, and $\phi_{\scT}(\cd)$ are defined as follows:
\begin{align}
\Th_{\scT}(t) &\deq -\[R(t)+D(t)^\top P_{\scT}(t)D(t)\]^{-1}
\[B(t)^\top P_{\scT}(t) +D(t)^\top P_{\scT}(t)C(t) +S(t)\], \label{MFLQ:Th-rep}\\
\wh\Th_{\scT}(t) &\deq -\[\hR(t)+\hD(t)^\top P_{\scT}(t)\hD(t)\]^{-1}
\[\hB(t)^\top\vP_{\scT}(t)+\hD(t)^\top P_{\scT}(t)\hC(t) +\hS(t)\], \label{MFLQ:whTH-rep}\\
\phi_{\scT}(t) &\deq -\[\hR(t)+\hD(t)^\top P_{\scT}(t)\hD(t)\]^{-1}\[\hB(t)^\top\f_{\scT}(t)+\hD(t)^\top P_{\scT}(t)\si(t)+r(t)\], \label{MFLQ:phi-rep}
\end{align}
with $\f_{\scT}(\cd)$ being the solution to the following ODE:
\begin{equation}\label{MF:bareta-d}\left\{\begin{aligned}
& \dot\f_{\scT}(t) +\[\hA(t)+\hB(t)\wh\Th_{\scT}(t)\]^\top\f_{\scT}(t)
+\[\hC(t)+\hD(t)\wh\Th_{\scT}(t)\]^\top P_{\scT}(t)\si(t) \\
&\hp{\dot\f_{\scT}(t)} + \wh\Th_{\scT}(t)^\top r(t) +\vP_{\scT}(t) b(t)+ q(t) =0, \q t\in[0,T],\\
& \f_{\scT}(T)=0.
\end{aligned}\right.\end{equation}
\end{lemma}

\section{Stability, stabilizability, and detectability}\label{Sect:Stable-detectable}

In this section, we introduce the concepts of mean-square exponential stability
and stabilizability for stochastic linear systems, as well as exact detectability
when observation variables exist.
As shown in the subsequent sections, these concepts are crucial to our discussion
of the turnpike property.

\ms

Denote by $[A(\cd),C(\cd)]$ the uncontrolled stochastic linear system
\begin{align}\label{SDE:AC}
dX(t) = A(t)X(t)dt+C(t)X(t)dW(t),
\end{align}
and by $[A(\cd),C(\cd)\nid E(\cd)]$ the following system with an observation process $Y(\cd)$:
\begin{equation}\label{SDE:XY}\lt\{\begin{aligned}
dX(t) &= A(t)X(t)dt+C(t)X(t)dW(t), \\
 Y(t) &= E(t)X(t),
\end{aligned}\rt.\end{equation}
where $E(\cd)\in L^\i(0,\i;\dbR^{l\times n})$ is a $\t$-periodic function.
The solution of \rf{SDE:AC} with initial state $x\in\dbR^n$ is denoted by $X(\cd\,;x)$,
and the corresponding observation process is denoted by $Y(\cd\,;x)$.
Let $\F(\cd)$ be the solution to the following matrix SDE:
$$ d\F(t)=A(t)\F(t)dt+C(t)\F(t)dW(t), \q \F(0)=I_{n}, \q t\ges0. $$
Then the state process $X(\cd\,;x)$ and the observation process $Y(\cd\,;x)$ can be represented as follows:
$$ X(t;x)=\F(t)x, \q Y(t;x)=E(t)\F(t)x, \q\forall t\ges 0. $$

\begin{definition}\label{Def:AC-stable}
The system $[A(\cd),C(\cd)]$ is called {\it mean-square exponentially stable}
if there exist constants $K,\l> 0$ such that
$$ \dbE|\F(t)|^{2}\les Ke^{-\l t}, \q\forall t\ges 0. $$
\end{definition}

\begin{remark}\label{remark:Sun-Yong2024}
It was shown in \cite[Proposition 3.3]{Sun-Yong2024:JDE} that $[A(\cd),C(\cd)]$
is mean-square exponentially stable if and only if for each (or equivalently, for some)
$\t$-periodic, uniformly positive definite function $\L(\cd)\in L^\i(0,\i;\dbS_+^n)$,
the Lyapunov differential equation
\begin{align}\label{P:stable}
\dot P(t) + P(t)A(t) + A(t)^\top P(t) + C(t)^\top P(t)C(t) + \L(t) =0
\end{align}
admits a $\t$-periodic, uniformly positive definite solution $P(\cd)\in C([0,\i);\dbS_+^n)$.
\end{remark}

\begin{definition}\label{Def:detactable}
The system $[A(\cd),C(\cd)\nid E(\cd)]$ is called {\it exactly detectable}
if
$$ \lim_{t\to\i}\dbE|X(t;x)|^2=0 $$
for any initial state $x\in\dbR^n$ satisfying
$$ Y(t;x) = 0, \q\as $$
for almost every $t\ges0$.
\end{definition}

\begin{remark}\label{remark:detec}
If $[A(\cd),C(\cd)]$ is mean-square exponentially stable, then by definition,
$[A(\cd),C(\cd)\nid E(\cd)]$ is exactly detectable for all $E(\cd)\in L^\i(0,\i;\dbR^{l\times n})$.
Conversely, if $l\ges n$ and $E(t)$ has rank $n$ for almost every $t\ges0$,
then $[A(\cd),C(\cd)\nid E(\cd)]$ is exactly detectable.
\end{remark}

The following result establishes a connection between mean-square exponential stability
and exact detectability. For the proof, we refer the reader to \cite[Theorem 3.1]{Dragon-etl2021}.

\begin{proposition}\label{prop:detec-stable}
Suppose that $[A(\cd),C(\cd)\nid E(\cd)]$ is exactly detectable. If for
$$
\L(\cd)\deq E(\cd)^\top E(\cd),
$$
equation \rf{P:stable} admits a $\t$-periodic solution $P(\cd)\in C([0,\i);\bar{\dbS}_+^n)$,
then $[A(\cd),C(\cd)]$ is mean-square exponentially stable.
\end{proposition}

Next, we consider the following controlled linear system, which we denote by
$[A(\cd),C(\cd);B(\cd),D(\cd)]$ for simplicity:
\begin{equation}\label{sys:ABCD}
dX(t) = [A(t)X(t)+B(t)u(t)]dt + [C(t)X(t)+D(t)u(t)]dW(t), \q t\ges0.
\end{equation}
When $C(\cd)=0$ and $D(\cd)=0$, the above system $[A(\cd),0\,;B(\cd),0]$
degenerates into a controlled linear ODE system.

\begin{definition}
The system $[A(\cd),C(\cd);B(\cd),D(\cd)]$ is called {\it mean-square exponentially stabilizable}
if there exists a $\t$-periodic function $\Th(\cd)\in L^{\i}(0,\i;\dbR^{m\times n})$ such
that the uncontrolled system $[A(\cd)+B(\cd)\Th(\cd),C(\cd)+D(\cd)\Th(\cd)]$ is mean-square
exponentially stable.
In this case, $\Th(\cd)$ is called a {\it stabilizer} of $[A(\cd),C(\cd);B(\cd),D(\cd)]$.
\end{definition}

In preparation for establishing a connection between mean-square exponential stabilizability
and exact detectability, we first present the following lemma.

\begin{lemma}\label{lmm:ED-th}
Let $N(\cd)\in L^\i(0,\i;\dbS^m)$ be a $\t$-periodic, uniformly positive definite function.
Let $\Th(\cd)\in L^\i(0,\i;\dbR^{m\times n})$ also be $\t$-periodic,
and set $E_{\scp\Th}(t)\deq\begin{pmatrix}E(t)\\ N(t)^{1/2}\Th(t)\end{pmatrix}$.
If $[A(\cd),C(\cd)\nid E(\cd)]$ is exactly detectable,
then $[A(\cd)+B(\cd)\Th(\cd),C(\cd)+D(\cd)\Th(\cd)\nid E_{\scp\Th}(\cd)]$ is also exactly detectable.
\end{lemma}

\begin{proof}
Denote by $X_{\scp\Th}(\cd\,;x)$ the solution of
\begin{equation}\label{SDE:XTh}\lt\{\begin{aligned}
dX_{\scp\Th}(t) &= [A(t)+B(t)\Th(t)]X_{\scp\Th}(t)dt+[C(t)+D(t)\Th(t)]X_{\scp\Th}(t)dW(t), \q t\ges0, \\
 X_{\scp\Th}(0) &= x.
\end{aligned}\right.\end{equation}
If $E_{\scp\Th}(t)X_{\scp\Th}(t;x)=0$ $\as$ for $\ae~t\ges 0$, then
$$ E(t)X_{\scp\Th}(t;x)=0, \q N(t)^{1/2}\Th(t)X_{\scp\Th}(t;x)=0, \q\as~\ae~t\ges 0, $$
which is equivalent to
\begin{align}\label{E-ETh}
E(t)X_{\scp\Th}(t;x)=0, \q \Th(t)X_{\scp\Th}(t;x)=0, \q\as~\ae~t\ges 0,
\end{align}
since $N(t)$ is positive definite for $\ae~t\ges 0$.
Consequently, the SDE \rf{SDE:XTh} reduces to
$$\lt\{\begin{aligned}
dX_{\scp\Th}(t) &= A(t)X_{\scp\Th}(t)dt + C(t)X_{\scp\Th}(t)dW(t), \q t\ges0, \\
 X_{\scp\Th}(0) &= x.
\end{aligned}\right.$$
Because $X_{\scp\Th}(\cd\,;x)$ and $X(\cd\,;x) $ satisfy the same SDE,
we must have $X_{\scp\Th}(\cd\,;x)=X(\cd\,;x)$. By the first equation in \rf{E-ETh}
and the exact detectability of $[A(\cd),C(\cd)\nid E(\cd)]$, we obtain
$$\lim_{t\to\i}\dbE|X_{\scp\Th}(t;x)|^2=\lim_{t\to\i}\dbE|X(t;x)|^2=0.$$
This proves the exact detectability of $[A(\cd)+B(\cd)\Th(\cd),C(\cd)+D(\cd)\Th(\cd)\nid E_{\scp\Th}(\cd)]$.
\end{proof}

We now establish a connection between mean-square exponential stabilizability
and exact detectability.

\begin{proposition}\label{prop:P-inf}
Suppose that $[A(\cd),C(\cd)\nid E(\cd)]$ is exactly detectable, and let $M(\cd)\deq E(\cd)^{\top}E(\cd)$.
Then the system $[A(\cd),C(\cd);B(\cd),D(\cd)]$ is mean-square exponentially stabilizable if and only if
for some (or equivalently, for any) $\t$-periodic, uniformly positive definite function $N(\cd)\in L^\i(0,\i;\dbS^m)$,
the differential Riccati equation
\begin{equation}\label{Sect3:Riccati-P}\begin{aligned}
&\dot{P}+PA+A^{\top}P+C^{\top}PC+M \\
&\hp{P} -(PB+C^{\top}PD)(N+D^{\top}PD)^{-1}(B^{\top}P+D^{\top}PC)=0
\end{aligned}\end{equation}
admits a unique $\t$-periodic solution $P(\cd)\in C([0,\i);\bar\dbS_+^n)$.
In this case, the function $\Th(\cd)$ defined by
\begin{equation}\label{Sect3:Stabilizer}
\Th(t)\deq-[N(t)+D(t)^{\top}P(t)D(t)]^{-1}[B(t)^{\top}P(t)+D(t)^{\top}P(t)C(t)]
\end{equation}
is a stabilizer of $[A(\cd),C(\cd);B(\cd),D(\cd)]$.
\end{proposition}

\begin{proof}
{\it Sufficiency.}
Suppose that \rf{Sect3:Riccati-P} admits a $\t$-periodic solution $P(\cd)\in C([0,\i);\bar\dbS_+^n)$.
Substituting \rf{Sect3:Stabilizer} into \rf{Sect3:Riccati-P} yields
$$
\dot{P}+(A+B\Th)^{\top}P+P(A+B\Th)+(C+D\Th)^{\top}P(C+D\Th)+M+\Th^{\top}N\Th=0.
$$
Let $E_{\scp\Th}(t)\deq(E(t)^{\top},\Th(t)^{\top}N(t)^{1/2})^{\top}$.
Then, by \autoref{lmm:ED-th}, $[A(\cd)+B(\cd)\Th(\cd),C(\cd)+D(\cd)\Th(\cd)\nid E_{\scp\Th}(\cd)]$
is exactly detectable.
Further, since
$$
E_{\scp\Th}(t)^{\top}E_{\scp\Th}(t)=M(t)+\Th(t)^{\top}N(t)\Th(t),
$$
it follows from \autoref{prop:detec-stable} that $[A(\cd)+B(\cd)\Th(\cd),C(\cd)+D(\cd)\Th(\cd)]$
is mean-square exponentially stable.

\ms
	
{\it Necessity.}
Let $\Th_0(\cd)$ be a stabilizer of system $[A(\cd),C(\cd);B(\cd),D(\cd)]$
so that $[A(\cd)+B(\cd)\Th_0(\cd),C(\cd)+D(\cd)\Th_0(\cd)]$ is mean-square exponentially stable.
Then by \cite[Proposition 3.3(i)]{Sun-Yong2024:JDE}, the differential equation
\begin{equation}\label{Sect3:Proof2}
\dot P_0+P_0(A+B\Th_0)+(A+B\Th_0)^{\top}P_0+(C+D\Th_0)^{\top}P_0(C+D\Th_0)+M+\Th_0^{\top}N\Th_0=0
\end{equation}
admits a unique $\t$-periodic solution $P_0(\cd)\in C([0,\i);\bar\dbS_+^n)$. Set
$$
\Th_1(t)\deq -[N(t)+D(t)^{\top}P_0(t)D(t)]^{-1}[B(t)^{\top}P_0(t)+D(t)^{\top}P_0(t)C(t)].
$$
Then \rf{Sect3:Proof2} can be written as
\begin{equation}\label{Sect3:Proof3}
\begin{aligned}
&\dot P_0 + P_0(A+B\Th_1) + (A+B\Th_1)^{\top}P_0 + (C+D\Th_1)^{\top}P_0(C+D\Th_1) \\
&\hp{P_0} + M + \Th_1^{\top}N\Th_1+(\Th_1-\Th_0)^\top(N+D^{\top}P_0D)(\Th_1-\Th_0)=0.
\end{aligned}
\end{equation}
Set
$$
F_1(t) \deq \begin{pmatrix*}E(t)\\N(t)^{1/2}\Th_1(t)\end{pmatrix*},\q
E_{\scp\Th_1}(t) \deq
\begin{pmatrix*}F_1(t) \\ [N(t)+D(t)^{\top}P_0(t)D(t)]^{1/2}[\Th_1(t)-\Th_0(t)]\end{pmatrix*}.
$$
Since $[A(\cd)+B(\cd)\Th_{0}(\cd),C(\cd)+D(\cd)\Th_{0}(\cd)]$ is mean-square
exponentially stable, $[A(\cd)+B(\cd)\Th_{0}(\cd),C(\cd)+D(\cd)\Th_{0}(\cd)\nid F_1(\cd)]$
is exactly detectable by \autoref{remark:detec}. Observing that
$$
A+B\Th_1=A+B\Th_{0}+B(\Th_1-\Th_0), \q C+D\Th_1=C+D\Th_0+D(\Th_1-\Th_0),
$$
we see from \autoref{lmm:ED-th} that
$[A(\cd)+B(\cd)\Th_{1}(\cd),C(\cd)+D(\cd)\Th_{1}(\cd)\nid E_{\scp\Th_1}(\cd)]$
is also exactly detectable.
Further, noting that
$$
E_{\scp\Th_1}^\top E_{\scp\Th_1} = M + \Th_1^{\top}N\Th_1+(\Th_1-\Th_0)^\top(N+D^{\top}P_0D)(\Th_1-\Th_0),
$$
we obtain by \eqref{Sect3:Proof3} and \autoref{prop:detec-stable} that
$[A(\cd)+B(\cd)\Th_{1}(\cd),C(\cd)+D(\cd)\Th_1(\cd)]$ is mean-square exponentially stable.
Inductively, for $i=1,2,\cdots$, set
\begin{align*}
& \Th_i(t) \deq -[N(t)+D(t)^{\top}P_{i-1}(t)D(t)]^{-1}[B(t)^{\top}P_{i-1}(t)+D(t)^{\top}P_{i-1}(t)C(t)], \\
&   A_i(t) \deq A(t)+B(t)\Th_i(t), \q C_i(t) \deq C(t)+D(t)\Th_i(t),
\end{align*}
and let $P_i(\cd)\in C([0,\i);\bar\dbS_+^n)$ be the unique $\t$-periodic
 positive semi-definite solution to the following equation:
\begin{equation}\label{Sect3:Proof4}
\dot{P}_i + P_iA_i + A_i^{\top}P_i + C_i^{\top}P_iC_i + M + \Th_i^{\top}N\Th_i = 0.
\end{equation}
A similar argument shows that $[A_i(\cd),C_i(\cd)]$ is mean-square
exponentially stable for all $i\ges1$.
Then, proceeding similarly to the proof of Proposition 3.8 in \cite{Sun-Yong2024:JDE},
we can further show that the two sequences $\{P_i(\cd)\}$ and $\{\Th_i(\cd)\}$ are
pointwise convergent, with the limit of $\{P_i(\cd)\}$ being the unique $\t$-periodic
positive semi-definite solution to \rf{Sect3:Riccati-P} and the limit of
$\{\Th_i(\cd)\}$ being a stabilizer of $[A(\cd),C(\cd);B(\cd),D(\cd)]$.
\end{proof}

\section{Exponential turnpike property of Riccati equations}\label{Sec:SRE}

Recall from \autoref{lmm:Sun2017} that under \ref{A1}--\ref{A2}, the system of
differential Riccati equations \rf{DRE:P+vP} admits a unique positive semi-definite
solution pair $(P_{\scT}(\cd),\vP_{\scT}(\cd))$.
It is shown in Sun--Yong \cite{Sun-Yong2024:JDE} that when the system $[A(\cd),C(\cd);B(\cd),D(\cd)]$
is mean-square exponentially stabilizable and
\begin{align}\label{cdtn:Sun-Yong2024}
Q(\cd)-S(\cd)^{\top}R(\cd)^{-1}S(\cd)
\end{align}
is uniformly positive definite, the differential Riccati equation
\begin{equation}\label{RE:Inf-P}\begin{aligned}
&\dot{P}+PA+A^{\top}P+C^{\top}PC+Q \\
&\hp{P} -(PB+C^{\top}PD+S^{\top})(R+D^{\top}PD)^{-1}(B^{\top}P+D^{\top}PC+S)=0
\end{aligned}\end{equation}
admits a unique $\t$-periodic, uniformly positive definite solution $P(\cd)\in C([0,\i);\dbS_+^n)$.
Moreover, $P_{\scT}(\cd)$ has the following exponential turnpike property:
For some constants $K,\l>0$ independent of $T$,
$$
|P_{\scT}(t)-P(t)| \les Ke^{-\l(T-t)}, \q\forall t\in[0,T].
$$
In this section, we introduce a condition weaker than \rf{cdtn:Sun-Yong2024}
and show that not only $P_{\scT}(\cd)$ but also $\vP_{\scT}(\cd)$ exhibits
the exponential turnpike property.

\ms

Let us introduce the condition first.

\begin{taggedassumption}{(A3)}\label{A3}
The systems
$$
[A(\cd),C(\cd);B(\cd),D(\cd)] \q\text{and}\q [\hA(\cd),0;\hB(\cd),0]
$$
are both mean-square exponentially stabilizable. Additionally, the systems
$$
[(A-BR^{-1}S)(\cd),(C-DR^{-1}S)(\cd)\nid(Q-S^{\top}R^{-1}S)^{1/2}(\cd)]
$$
and
$$
[(\hA-\hB\hR^{-1}\hS)(\cd),0\nid(\hQ-\hS^{\top}\hR^{-1}\hS)^{1/2}(\cd)]
$$
are both exactly detectable.
\end{taggedassumption}

\begin{remark}
Clearly, when \rf{cdtn:Sun-Yong2024} holds, the system $[(A-BR^{-1}S)(\cd),(C-DR^{-1}S)(\cd)\nid(Q-S^{\top}R^{-1}S)^{1/2}(\cd)]$
is exactly detectable. However, the converse is not true in general.
\end{remark}

\begin{proposition}\label{prop:periodi-Ric}
Let {\rm\ref{A1}--\ref{A3}} hold. Then the differential Riccati equation
\rf{RE:Inf-P} admits a unique $\t$-periodic positive semi-definite solution
$P(\cd)\in C([0,\i);\bar\dbS_+^n)$, and the function $\Th(\cd)$ defined by
\begin{equation}\label{S4:Stabilizer}
\Th(t)\deq-[R(t)+D(t)^{\top}P(t)D(t)]^{-1}[B(t)^{\top}P(t)+D(t)^{\top}P(t)C(t)+S(t)]
\end{equation}
is a stabilizer of $[A(\cd),C(\cd);B(\cd),D(\cd)]$.
Moreover, the differential Riccati equation
\begin{equation}\label{RE:Inf-vP}\begin{aligned}
&\dot{\vP}+\vP\hA+\hA^{\top}\vP+\hC^{\top}P\hC+\hQ\\
&\hp{\vP} -(\vP\hB+\hC^{\top}P\hD+\hS^\top)(\hR+\hD^{\top}P\hD)^{-1}(\hB^\top\vP+\hD^{\top}P\hC+\hS)=0
\end{aligned}\end{equation}
also admits a unique $\t$-periodic positive semi-definite solution $\vP(\cd)\in C([0,\i);\bar\dbS_+^n)$,
and the function $\wh\Th(\cd)$ defined by
\begin{equation}\label{STBL:whTh}
\wh\Th(t)\deq -[\hR(t)+\hD(t)^{\top}P(t)\hD(t)]^{-1}[\hB(t)^{\top}\vP(t)+\hD(t)^{\top}P(t)\hC(t)+\hS(t)]
\end{equation}
is a stabilizer of $[\hA(\cd),0\,;\hB(\cd),0]$.
\end{proposition}

\begin{proof}
Since $[A(\cd),C(\cd);B(\cd),D(\cd)]$ is mean-square exponentially stabilizable,
so is $[(A-BR^{-1}S)(\cd),(C-DR^{-1}S)(\cd);B(\cd),D(\cd)]$.
Further, since $[(A-BR^{-1}S)(\cd),(C-DR^{-1}S)(\cd)\nid(Q-S^{\top}R^{-1}S)^{1/2}(\cd)]$
is exactly detectable, we have by \autoref{prop:P-inf} that the differential equation
\begin{equation}\label{proof-prop4.2-1}\begin{aligned}
&\dot{P}+P(A-BR^{-1}S)+(A-BR^{-1}S)^{\top}P \\
&\hp{P} +(C-DR^{-1}S)^{\top}P(C-DR^{-1}S)+ Q-S^{\top}R^{-1}S \\
&\hp{P} -[PB+(C-DR^{-1}S)^{\top}PD](R+D^{\top}PD)^{-1}[B^{\top}P+D^{\top}P(C-DR^{-1}S)]=0
\end{aligned}\end{equation}
admits a unique $\t$-periodic solution $P(\cd)\in C([0,\i);\bar\dbS_+^n)$,
and the function $\G(\cd)$ defined by
\begin{align*}
\G &\deq-(R+D^{\top}PD)^{-1}[B^{\top}P+D^{\top}P(C-DR^{-1}S)] = \Th + R^{-1}S
\end{align*}
is a stabilizer of $[(A-BR^{-1}S)(\cd),(C-DR^{-1}S)(\cd);B(\cd),D(\cd)]$.
By straightforward simplification, it can be seen that \rf{proof-prop4.2-1} is exactly \rf{RE:Inf-P},
and that the function $\Th(\cd)$ defined by \rf{S4:Stabilizer} is a stabilizer
of $[A(\cd),C(\cd);B(\cd),D(\cd)]$.

\ms

Next, we let
\begin{alignat*}{4}
\wh\cQ(\cd) &\deq \hC(\cd)^{\top}P(\cd)\hC(\cd)+\hQ(\cd), &\q
\wh\cS(\cd) &\deq \hD(\cd)^{\top}P(\cd)\hC(\cd)+\hS(\cd), \\
\wh\cR(\cd) &\deq \hR(\cd)+\hD(\cd)^{\top}P(\cd)\hD(\cd), &\q
\wh\cH(\cd) &\deq \wh\cQ(\cd)-\wh\cS(\cd)^{\top}\wh\cR(\cd)^{-1}\wh\cS(\cd).
\end{alignat*}
It is straightforward to verify that
\begin{align*}
\wh\cH &= \hQ-\hS^{\top}\hR^{-1}\hS +(\hC-\hD\wh\cR^{-1}\wh\cS)^{\top}P(\hC-\hD\wh\cR^{-1}\wh\cS) \\
&\hp{=\ } +(\hR^{-1}\hS- \wh\cR^{-1}\wh\cS)^{\top} \hR(\hR^{-1}\hS-\wh\cR^{-1}\wh\cS)  \\
&= \cE^{\top}\cE,
\end{align*}
where
$$
\cE(\cd) \deq \begin{pmatrix*}
\big[\hQ(\cd)-\hS(\cd)^{\top}\hR(\cd)^{-1}\hS(\cd)\big]^{1/2} \\
P(\cd)^{1/2}\big[\hC(\cd)-\hD(\cd)\wh\cR(\cd)^{-1}\wh\cS(\cd)\big] \\
\hR(\cd)^{1/2}\big[\hR(\cd)^{-1}\hS(\cd)-\wh\cR(\cd)^{-1}\wh\cS(\cd)\big] \end{pmatrix*}.
$$
Note that
\begin{align*}
\hA-\hB\wh\cR^{-1}\wh\cS = \hA-\hB\hR^{-1}\hS +
\begin{pmatrix*}0 & \hB\end{pmatrix*}
\begin{pmatrix*}P^{1/2}(\hC-\hD\wh\cR^{-1}\wh\cS) \\ \hR^{-1}\hS-\wh\cR^{-1}\wh\cS\end{pmatrix*}.
\end{align*}
Then, by replacing the matrix functions $N$, $\Th$, $E$, $A$, $C$, $B$, and $D$ in \autoref{lmm:ED-th} with
$\diag(I_n,\hR)$, $\begin{pmatrix*}P^{1/2}(\hC-\hD\wh\cR^{-1}\wh\cS) \\ \hR^{-1}\hS-\wh\cR^{-1}\wh\cS\end{pmatrix*}$,
$\big(\hQ-\hS^{\top}\hR^{-1}\hS\big)^{1/2}$, $\hA-\hB\hR^{-1}\hS$, $0$,
$\begin{pmatrix*}0 & \hB\end{pmatrix*}$, and $0$, respectively,
we see that $[(\hA-\hB\cR^{-1}\wh\cS)(\cd),0\nid \cE(\cd)]$ is exactly detectable.
Now applying \autoref{prop:P-inf} to $[(\hA-\hB\cR^{-1}\wh\cS)(\cd),0\nid \cE(\cd)]$
and proceeding similarly to the previous proof for $P(\cd)$, we obtain the
desired conclusion for $\vP(\cd)$.
\end{proof}

In the rest of this section, let $(P_{\scT}(\cd),\vP_{\scT}(\cd))$ be the unique
solution pair of \rf{DRE:P+vP} and $(P(\cd),\vP(\cd))$ be as in \autoref{prop:periodi-Ric}.
Using a similar argument employed in the proof of \cite[Theorem 5.6]{Sun-Yong2024:JDE},
we can obtain the following result, noting that the condition here is slightly weaker
than that of \cite[Theorem 5.6]{Sun-Yong2024:JDE}.

\begin{proposition}\label{prop:PT-P}
Let {\rm\ref{A1}--\ref{A3}} hold. There exist constants $K,\l>0$, independent of $T$,
such that
$$
\big|P_{\scT}(t)-P(t)\big| \les Ke^{-\l(T-t)},\q\forall t\in[0,T].
$$
\end{proposition}

Next, we shall show that a similar estimate holds for $|\vP_{\scT}(\cd)-\vP(\cd)|$.
In preparation, we present the following result first.

\begin{proposition}\label{prop:finite--P+vP}
Let {\rm\ref{A1}--\ref{A3}} hold.
\begin{enumerate}[\rm(i)]
\item $\vP_{\scT}(\cd)$ is nondecreasing with respect to $T$, that is,
      $$
      \vP_{\scT_1}(t)\les \vP_{\scT_2}(t), \q \forall 0\les t\les T_1\les T_2 <\i.
      $$
\item For any $0\les t \les T <+\i$,
      $$
      \vP_{\scT+\tau}(t+\t)=\vP_{\scT}(t).
      $$
\item $\lim_{T\to\i}\vP_{\scT}(t)=\vP(t)$ for all $t\ges0$.
\end{enumerate}
\end{proposition}

\begin{proof}
First, using similar arguments as in the proofs of \cite[Propositions 5.2 and 5.4]{Sun-Yong2024:JDE},
we can easily show that $\{P_{\scT}(\cd)\}_{\scT\ges 0}$ satisfies the same properties in this
proposition.

\ms

(i) For any $T_1\les T_2<\i$, set
$$\bar P(t)\deq P_{\scT_2}(t)-P_{\scT_1}(t),\q \bar\vP(t)\deq \vP_{\scT_2}(t)-\vP_{\scT_1}(t), \q \forall t\in[0,T_1],$$
and define
$$\wh\Th_{i}(t)\deq -[\hR(t)+\hD(t)^{\top}P_{\scT_i}(t)\hD(t)]^{-1}
 [\hB(t)^{\top}\vP_{\scT_i}(t)+\hD(t)^{\top}P_{\scT_i}(t)\hC(t)+\hS(t)],
\q i=1,2.
$$
Then, $\bar\vP(t)$ satisfies the following ODE:
\begin{equation*}
\lt\{\begin{aligned}
&\dot{\bar\vP}+ (\hA+\hB\wh\Th_2)^{\top} \bar\vP + \bar\vP (\hA+\hB\wh\Th_2) + \G=0, \q t\in[0,T_1],\\
&\vP(T_1)\ges 0,
\end{aligned}\rt.\end{equation*}
where
\begin{align*}
\G &\deq \hC^{\top}\bar P\hC - \wh\Th_{2}^{\top}\hB^{\top}\bar\vP - \bar\vP\hB\wh\Th_2
-\wh\Th_2^{\top}(\hR+\hD^{\top}P_{\scT_2}\hD)\wh\Th_2 + \wh\Th_1^{\top}(\hR+\hD^{\top}P_{\scT_1}\hD)\wh\Th_1\\
&=(\hC+\hD\wh\Th_2)^{\top}\bar P(\hC+\hD\wh\Th_2) -(\bar\vP\hB+\hC^{\top}\bar P\hD)\wh\Th_2
-\wh\Th_2^{\top}(\hB^{\top}\bar\vP+\hD^{\top}\bar P\hC)\\
&\hp{=\ } -\wh\Th_2^{\top}\hD^{\top}\bar P\hD\wh\Th_2
-\wh\Th_2^{\top}(\hR+\hD^{\top}P_{\scT_2}\hD)\wh\Th_2 + \wh\Th_1^{\top}(\hR+\hD^{\top}P_{\scT_1}\hD)\wh\Th_1\\
&=(\hC+\hD\wh\Th_2)^{\top}\bar P(\hC+\hD\wh\Th_2)- \wh\Th_2^{\top}(\hR+\hD^{\top}P_{\scT_1}\hD)\wh\Th_1
-\wh\Th_1^{\top}(\hR+\hD^{\top}P_{\scT_1}\hD)\wh\Th_2 \\
&\hp{=\ } + \wh\Th_2^{\top}(\hR+\hD^{\top}P_{\scT_1}\hD)\wh\Th_2 + \wh\Th_1^{\top}(\hR+\hD^{\top}P_{\scT_1}\hD)\wh\Th_1 \\
&=(\hC+\hD\wh\Th_2)^{\top}\bar P(\hC+\hD\wh\Th_2) + (\wh\Th_2-\wh\Th_1)^{\top}(\hR+\hD^{\top}P_{\scT_1}\hD)(\wh\Th_2-\wh\Th_1).
\end{align*}
Obviously, $\G(t)\ges0$ for all $t\in [0,T_1]$, since $P_{\scT}(t)$ is nondecreasing with respect to $T$.
Now, let $\wh\F(\cd)$ be the unique solution of
$$\lt\{\begin{aligned}
d\wh\F(t) &= [\hA(t)+\hB(t)\wh\Th_2(t)]\wh\F(t)dt, \q t\ges 0,\\
 \wh\F(0) &= I_n.\end{aligned}\rt.
$$
We have by the variation of constants formula that for any $t\in[0,T_1]$,
\begin{align*}
\bar\vP(t) &= \[\wh\F(T_1)\wh\F(t)^{-1}\]^{\top}\bar\vP(T_1)\[\wh\F(T_1)\wh\F(t)^{-1}\]\\
&\hp{=\ }+ \int_t^{T_1}\[\wh\F(s)\wh\F(t)^{-1}\]^{\top}\G(s)\[\wh\F(s)\wh\F(t)^{-1}\]ds \ges0.
\end{align*}

(ii) Set $\wh\vP_{\scT}(t)\deq \vP_{\scT+\tau}(t+\t)$. Then $\wh\vP_{\scT}(T)=0$.
Using \rf{DRE:P+vP} and the fact $P_{\scT+\tau}(t+\t)=P_{\scT}(t)$, we obtain
\begin{align*}
-\dot{\wh\vP}_{\scT}(t)&=-\dot\vP_{\scT+\t}(t+\t)\\
&=\wh\vP_{\scT}\hA+\hA^{\top}\wh\vP_{\scT}
+\hC^{\top}P_{\scT}\hC+\hQ \\
&\hp{=\ }-(\wh\vP_{\scT}\hB+\hC^\top P_{\scT}\hD+\hS^\top)
(\hR+\hD^\top P_{\scT}\hD)^{-1}(\hB^\top\wh\vP_{\scT}+\hD^\top P_{\scT}\hC +\hS).
\end{align*}
Thus, $\wh\vP_{\scT}(\cd)$ satisfies the same ODE as $\vP_{\scT}(\cd)$.
By uniqueness, $\wh\vP_{\scT}(\cd)=\vP_{\scT}(\cd)$.

\ms

(iii) In the proof of (i), by replacing the matrix functions $\vP_{\scT_1}(\cd)$ and
$\vP_{\scT_2}(\cd)$ with $\vP_{\scT}(\cd)$ and $\vP(\cd)$, respectively, we get
$$ \vP_{\scT}(t)\les \vP(t), \q \forall 0\les t\les T<\i.$$
Then, by the monotone convergence theorem, the limit
$$
\vP_{\scp\i}(t)\deq \lim_{\scT\to\i}\vP_{\scT}(t)
$$
exists for all $t\in[0,\i)$. Furthermore, we have from (ii) that
$$
\vP_{\scp\i}(t+\t)
=\lim_{\scT\to\i}\vP_{\scT}(t+\t)
=\lim_{\scT\to\i}\vP_{\scT+\tau}(t+\t)
=\vP_{\scp\i}(t),\q \forall t\ges 0,
$$
which implies that $\vP_{\scp\i}(\cd)$ is $\t$-periodic and positive semi-definite.
On the other hand, for any $0\les s \les t \les T$,
\begin{align*}
\vP_{\scT}(s)-\vP_{\scT}(t)=&\int_{s}^{t}\[\vP_{\scT}\hA+\hA^{\top}\vP_{\scT} +
\hC^{\top}P_{\scT}\hC + \hQ - (\vP_{\scT}\hB+\hC^{\top}P_{\scT}\hD+\hS^{\top}) \\
&\q\q\times (\hR+\hD^{\top}P_{\scT}\hD)^{-1}(\hB^{\top}\vP_{\scT}+\hD^{\top}P_{\scT}\hC+\hS)\]dr.	
\end{align*}
Letting $T\to \i$, we obtain by the bounded convergence theorem that
\begin{align*}
\vP_{\scp\i}(s)-\vP_{\scp\i}(t)=&\int_{s}^{t}\[\vP_{\scp\i}\hA+\hA^{\top}\vP_{\scp\i} +
\hC^{\top}P\hC + \hQ - (\vP_{\scp\i}\hB+\hC^{\top}P\hD+\hS^{\top}) \\
&\q\q\times (\hR+\hD^{\top}P\hD)^{-1}(\hB^{\top}\vP_{\scp\i}+\hD^{\top}P\hC+\hS)\]dr,	
\end{align*}
which is exactly the integral version of \rf{RE:Inf-vP}.
By uniqueness of solutions, $\vP_{\scp\i}(t)=\vP(t)$ for all $t\ges0$.
\end{proof}

\begin{theorem}\label{thm:vP-vPT}
Let {\rm\ref{A1}--\ref{A3}} hold. Then there exist constants $K,\l>0$, independent of $T$, such that
$$
\big|\vP_{\scT}(t)-\vP(t)\big| \les Ke^{-\l(T-t)}, \q\forall t\in[0,T].
$$
\end{theorem}

\begin{proof}
Set
$$
\Si_{\scT}(t)\deq P(t)-P_{\scT}(t),\q \L_{\scT}(t)\deq \vP(t)-\vP_{\scT}(t), \q t\in[0,T].
$$
By \autoref{prop:finite--P+vP}, $\Si_{\scT}(t)\ges 0$ and $\L_{\scT}(t)\ges 0$ for all $t\in[0,T]$.
Let $\wh\Th(\cd)$ and $\wh\Th_{\scT}(\cd)$ be defined by \rf{S4:Stabilizer} and \rf{STBL:whTh},
respectively. Then
\begin{equation}\label{EQ:LT}
\dot{\L}_{\scT}(t)+\L_{\scT}(t)[\hA(t)+\hB(t)\wh\Th(t)]+[\hA(t)+\hB(t)\wh\Th(t)]^{\top}\L_{\scT}(t)+\G(t)=0,
\end{equation}
where
$$
\G\deq (\hC+\hD\wh\Th)^{\top}\Si_{\scT}(\hC+\hD\wh\Th)+
(\wh\Th-\wh\Th_{\scT})^{\top}(\hR+\hD^{\top}P_{\scT}\hD)(\wh\Th-\wh\Th_{\scT}).
$$
Observing that
\begin{align*}
\wh\Th-\wh\Th_{\scT}
&=-(\hR+\hD^{\top}P_{\scT}\hD)^{-1}[\hB^{\top}\L_{\scT}+\hD^{\top}\Si_{\scT}(\hC+\hD\wh\Th)],
\end{align*}
and that $\Si_{\scT}(\cd)$, $\L_{\scT}(\cd)$, and $P_{\scT}(\cd)$
are all bounded uniformly in $T$, we can choose a constant $K_1>0$, independent of $T$,
such that
\begin{equation}\label{INEQ:G}
\big|\G(t)\big|\les K_1\[\big|\L_{\scT}(t)\big|^2+\big|\Si_{\scT}(t)\big|\],\q\forall t\in[0,T].
\end{equation}
Since $\wh\Th(\cd)$ is a stabilizer of $[\hA(\cd),0;\hB(\cd),0]$, the ODE system
\begin{equation}\label{SDE:hF_Th}\lt\{\begin{aligned}
d\wh\F_{\scp\Th}(t)&=[\hA(t)+\hB(t)\wh\Th(t)]\wh\F_{\scp\Th}(t)dt, \q t\ges 0, \\
 \wh\F_{\scp\Th}(0)&=I_{n}
\end{aligned}\rt.\end{equation}
is mean-square exponentially stable. By \cite[Corollary 3.4]{Sun-Yong2024:JDE},
there exist constants $K_2,\l_{1}>0$ such that
\begin{equation}\label{INEQ:hF_Th}
\big|\wh\F_{\scp\Th}(t)\wh\F_{\scp\Th}(s)^{-1}\big|^2
\les K_2e^{-\l_1(t-s)},\q\forall 0\les s\les t<\i.
\end{equation}
Additionally, \autoref{prop:PT-P} implies that
\begin{equation}\label{INEQ:Si_T}
\big|\Si_{\scT}(t)\big|\les K_3e^{-2\l_2(T-t)},\q\forall t\in[0,T],
\end{equation}
for some constants $K_3,\l_2>0$ independent of $T$.
According to \autoref{prop:finite--P+vP},
$\L_{\scT}(\cd)$ is nonincreasing with $\lim_{\scT\to\i}\L_{\scT}(t)=0$.
So with $\l\deq\min\{\l_1,\l_2\}$, we can choose a constant $N>\t$ such that
\begin{equation}\label{INEQ:ep}
0<\e\deq K_2\L_{\scN}(0)+\frac{K_1K_2K_3}{\l}e^{-2\l N}\les \frac{\l}{2K_1K_2}.
\end{equation}
Now, let us turn to the estimation of $|\L_{\scT}(\cd)|$. First, if $T<N+\t$,
$$|\L_{\scT}(t)|\les |\L_{t}(t)|=|\vP(t)|,\q t\in [0,T].$$
Let $K_4\deq e^{\l(N+\t)}\max_{t\in[0,\t]}\vP(t)$. Then
\begin{equation}\label{INEQ:LtN}
|\L_{\scT}(t)|\les K_{4}e^{-\l(T-t)},\q\forall t\in[0,T].
\end{equation}
Next, we consider the case of $T\ges N+\t$. Let $T$ be fixed and
$$
k\deq \max\{k; k \text{ is an integer, and } N+k\t\les T\}.
$$
Then it follows from \autoref{prop:finite--P+vP} that
$$0\les \L_{\scT}(k\tau)\les \L_{N+k\tau}(k\tau)=\L_{N}(0).$$
Applying the variation of constants formula to \rf{EQ:LT}, we obtain
\begin{equation*}
\begin{aligned}
\L_{\scT}(t)&=\[\wh\F_{\scp\Th}(k\tau)\wh\F_{\scp\Th}(t)^{-1}\]^{\top}\L_{\scT}(k\tau)
\[\wh\F_{\scp\Th}(k\tau)\wh\F_{\scp\Th}(t)^{-1}\] \\
&\hp{=\ } +\int_{t}^{k\tau}\[\wh\F_{\scp\Th}(s)\wh\F_{\scp\Th}(t)^{-1}\]^{\top}\G(s)
\[\wh\F_{\scp\Th}(s)\wh\F_{\scp\Th}(t)^{-1}\]ds, \q\forall t\les k\t.
\end{aligned}
\end{equation*}
Combining \rf{INEQ:G}, \rf{INEQ:hF_Th}--\rf{INEQ:ep}, we further have
\begin{align*}
|\L_{\scT}(t)|
&\les K_2e^{-\l(k\t-t)}|\L_{\scT}(k\t)| + K_1K_2\int_{t}^{k\t}e^{-\l(s-t)}\[K_3e^{-2\l(T-s)}+|\L_{\scT}(s)|^2\]ds \\
&\les K_2|\L_{N}(0)| e^{-\l(k\t-t)} + \frac{K_1K_2K_3}{\l}e^{-\l(2T-t)}(e^{\l k\t}-e^{\l t}) \\
&\hp{=\ } +K_1K_2\int_{t}^{k\t}e^{-\l(s-t)}|\L_{\scT}(s)|^2ds \\
&\les \varepsilon e^{-\l(k\t-t)} + K_1K_2\int_{t}^{k\t}e^{-\l(s-t)}|\L_{\scT}(s)|^2ds, \q\forall t\in[0,k\t].
\end{align*}
Set $g(t)\deq K_1K_2e^{\l t}|\L_{\scT}(k\tau-t)|$. Then for $t\in[0,k\t]$,
$$
g(t)\les K_1K_2\e + \int_0^t e^{-\l s}g(s)^2ds \les \frac{\l}{2}+\int_{0}^{t}e^{-\l s}g(s)^2ds.
$$
It follows from \cite[Lemma 5.5]{Sun-Yong2024:JDE} that
$$g(t)\les \l,\q\forall t\in[0,k\t].$$
Consequently,
\begin{equation}\label{INEQ:0-kt}\begin{aligned}
|\L_{\scT}(t)|&=\frac{1}{K_1K_2}e^{-\l(k\tau-t)}g(k\tau-t)\les
\frac{\l}{K_1K_2}e^{\l(T-k\tau)}e^{-\l(T-t)}\\
&\les \frac{\l}{K_1K_2}e^{\l(N+\tau)}e^{-\l(T-t)},\q\forall t\in[0,k\t].
\end{aligned}\end{equation}
For $t\in[k\tau, T]$, we have
$$
0 \les \L_{\scT}(t) \les \L_{N+k\t}(t)=\L_{N}(t-k\t) \les \L_{\t}(t-k\t)\les \max_{t\in[0,\t]}\vP(t).
$$
Recalling that $T<N+(k+1)\t$ and $K_4\deq e^{\l(N+\t)}\max_{t\in[0,\t]}\vP(t)$, we obtain
\begin{equation}\label{INEQ:kt-T}
|\L_{\scT}(t)|\les K_4 e^{-\l(N+\t)}\les K_4 e^{-\l(N+(k+1)\t-k\t)}\les K_4 e^{-\l(T-t)},\q\forall t\in[k\t,T].
\end{equation}
Combining \rf{INEQ:LtN}--\rf{INEQ:kt-T} and taking $K\deq\max\{K_4,\frac{\l}{K_1K_2}e^{\l(N+\t)}\}$,
we get the desired result.
\end{proof}

As a consequence of \autoref{prop:PT-P} and \autoref{thm:vP-vPT}, we have the following result.

\begin{corollary}\label{CORO:Th-whTh}
Let {\rm\ref{A1}--\ref{A3}} hold. Let $\Th_{\scT}(\cd)$, $\wh\Th_{\scT}(\cd)$,  $\Th(\cd)$,
and $\wh\Th(\cd)$ be defined by \rf{MFLQ:Th-rep}, \rf{MFLQ:whTH-rep}, \rf{S4:Stabilizer},
and \rf{STBL:whTh}, respectively.
There exist constants $K,\l>0$, independent of $T$, such that
$$
|\Th(t)-\Th_{\scT}(t)|+|\wh\Th(t)-\wh\Th_{\scT}(t)|\les Ke^{-\l(T-t)},\q\forall t\in[0,T].
$$
\end{corollary}

\section{The periodic mean-field stochastic LQ optimal control problem}\label{Sec:PMFLQ}

In this section, we investigate a periodic mean-field stochastic LQ optimal control problem.
In the next section, we will see that, through $\t$-periodic extension and linear transformation,
the optimal pair of this problem leads to the turnpike limit of Problem (MFLQ)$_{\scT}$.

\ms

Let $P(\cd)$ and $\vP(\cd)$ be the unique $\t$-periodic positive semi-definite solutions
to \rf{RE:Inf-P} and \rf{RE:Inf-vP}, respectively, and let $\Th(\cd)$ and $\wh\Th(\cd)$
be the functions defined in \rf{S4:Stabilizer}, and \rf{STBL:whTh}, respectively.
For notational simplicity, let
\begin{equation}\label{PMFLQ:Notation1}\begin{aligned}
 \cA(\cd)     &\deq  A(\cd)+B(\cd)\Th(\cd),        ~
&\wh\cA(\cd)  &\deq  \hA(\cd)+\hB(\cd)\wh\Th(\cd), ~
&\bar\cA(\cd) &\deq  \wh\cA(\cd)-\cA(\cd), \\
 \cC(\cd)     &\deq  C(\cd)+D(\cd)\Th(\cd),        ~
&\wh\cC(\cd)  &\deq  \hC(\cd)+\hD(\cd)\wh\Th(\cd), ~
&\bar\cC(\cd) &\deq  \wh\cC(\cd)-\cC(\cd).
\end{aligned}\end{equation}
Consider the following controlled mean-field SDE over $[0,\t]$:
\begin{equation}\label{PMFLQ:X1}\left\{\begin{aligned}
d\cX(t) &= \{\cA(t)\cX(t)+\bar\cA(t)\dbE[\cX(t)]+B(t)v(t)+\bB(t)\dbE[v(t)]+b(t)\}dt\\
	    &\hp{=\ } +\{\cC(t)\cX(t)+\bar\cC(t)\dbE[\cX(t)]+D(t)v(t)+\bD(t)\dbE[v(t)]+\si(t)\}dW(t),\\
\mu_{\scp\cX(0)} &= \mu_{\scp\cX(\t)},
\end{aligned}\right.\end{equation}
where $\mu_{\scp\xi}$ denotes the distribution of a random variable $\xi$.

\ms

The following result establishes the well-posedness of the SDE \rf{PMFLQ:X1}.

\begin{proposition}\label{prop:X1-wellposed}
Let {\rm\ref{A1}--\ref{A3}} hold. For any $v(\cd)\in\sU[0,\t]$, the SDE \rf{PMFLQ:X1}
admits a unique (in the distribution sense) square integrable solution.
\end{proposition}

\begin{proof}
Let $\cP(\dbR^n)$ be the set of probability measures on $(\dbR^n,\sB(\dbR^n))$
having finite second moment.
We equip $\cP(\dbR^n)$ with the $L^2$-Wasserstein distance:
$$
d(\nu_1,\nu_2) \deq \inf\lt\{\sqrt{\dbE|\xi_1-\xi_2|^2}\Bigm|
\xi_i\text{ is a random variable in $\dbR^n$ with } \mu_{\scp\xi_i} = \nu_i; ~i=1,2 \rt\},
$$
where $\mu_{\scp\xi_i}$ denotes the distribution of $\xi_i$.
Thus, $\cP(\dbR^n)$ is a complete metric space.
Note that for any positive definite matrix $M,N>0$, the function $d_{\scp M,N}(\cd\,,\cd)$
defined by
\begin{align*}
d_{\scp M,N}(\nu_1,\nu_2)
&\deq \inf\Big\{\big[\dbE\lan M[\xi-\dbE(\xi)],\xi-\dbE(\xi)\ran + \lan N\dbE(\xi),\dbE(\xi)\ran\big]^{1/2}
\Bigm|\xi=\xi_1-\xi_2,\\
&\hp{=\inf\bigg\{ } \xi_i\text{ is a random variable in $\dbR^n$ with }\mu_{\scp\xi_i}=\nu_i;~i=1,2 \Big\}
\end{align*}
is also a distance on $\cP(\dbR^n)$, which is equivalent to $d(\cd\,,\cd)$.
Denote by $\cX(\cd\,;\nu)$ the solution (which is unique in the distribution sense) of
$$\left\{\begin{aligned}
d\cX(t) &= \{\cA(t)\cX(t)+\bar\cA(t)\dbE[\cX(t)]+B(t)v(t)+\bB(t)\dbE[v(t)]+b(t)\}dt\\
	    &\hp{=\ } +\{\cC(t)\cX(t)+\bar\cC(t)\dbE[\cX(t)]+D(t)v(t)+\bD(t)\dbE[v(t)]+\si(t)\}dW(t),\\
 \mu_{\scp\cX(0)} &= \nu\in\cP(\dbR^n).
\end{aligned}\right.$$
Then we can define a mapping $\cL: \cP(\dbR^{n})\mapsto \cP(\dbR^{n})$ by
$$
\cL(\nu) \deq \mu_{\scp X(\t;\nu)}.
$$
Next, we show that the $\cL$ is a contraction mapping.
To this end, let $\nu_i\in\cP(\dbR^n)$ $(i=1,2)$, and let $\xi_i$ be $\sF_0$-measurable
random variables with $\mu_{\scp\xi_i}=\nu_i$.
Denote by $\cX_i(\cd)$ the solution of
$$\left\{\begin{aligned}
d\cX(t) &= \{\cA(t)\cX(t)+\bar\cA(t)\dbE[\cX(t)]+B(t)v(t)+\bB(t)\dbE[v(t)]+b(t)\}dt\\
	    &\hp{=\ } +\{\cC(t)\cX(t)+\bar\cC(t)\dbE[\cX(t)]+D(t)v(t)+\bD(t)\dbE[v(t)]+\si(t)\}dW(t),\\
 \cX(0) &= \xi_i.
\end{aligned}\right.$$
Then $\wh\cX(\cd)\deq \cX_1(\cd)-\cX_2(\cd)$ satisfies
$$\left\{\begin{aligned}
d\wh\cX(t) &= \big\{\cA(t)\wh\cX(t)+\bar\cA(t)\dbE[\wh\cX(t)]\big\}dt
            + \big\{\cC(t)\wh\cX(t)+\bar\cC(t)\dbE[\wh\cX(t)]\big\}dW(t),\\
 \wh\cX(0) &= \xi_1-\xi_2 \equiv \h\xi.
\end{aligned}\right.$$
Clearly, $\dbE[\wh\cX(\cd)]$ and $\wh\cY(\cd)\deq \wh\cX(\cd)-\dbE[\wh\cX(\cd)]$ satisfies
$$\left\{\begin{aligned}
d\dbE[\wh\cX(t)] &= \wh\cA(t)\dbE[\wh\cX(t)]dt,\\
 \dbE[\wh\cX(0)] &= \dbE[\h\xi],
\end{aligned}\right.$$
and
$$\left\{\begin{aligned}
d\wh\cY(t) &= \cA(t)\wh\cY(t)dt  + \big\{\cC(t)\wh\cY(t)+\wh\cC(t)\dbE[\wh\cX(t)]\big\}dW(t),\\
 \wh\cY(0) &= \h\xi-\dbE[\h\xi],
\end{aligned}\right.$$
respectively. By \autoref{prop:periodi-Ric}, $[\cA(\cd),\cC(\cd)]$ and $[\wh\cA(\cd),0]$
are exponentially mean-square stable.
Thus, according to \autoref{remark:Sun-Yong2024}, there exists a unique $\t$-periodic,
uniformly positive definite function $M(\cd)\in C([0,\i);\dbS_+^n)$ such that
$$
\dot{M}(t)+M(t)\cA(t)+\cA(t)^{\top}M(t)+\cC(t)^{\top}M(t)\cC(t)+I_n=0,
$$
and there exists a unique $\t$-periodic, uniformly positive definite function
$N(\cd)\in C([0,\i);\dbS_+^n)$ such that
$$
\dot{N}(t)+N(t)\wh\cA(t)+\wh\cA(t)^{\top}N(t)+\wh\cC(t)^{\top}M(t)\wh\cC(t)+I_n=0.
$$
By It\^{o}'s rule,
\begin{align*}
{d\over dt}\dbE\lan M(t)\wh\cY(t),\wh\cY(t)\ran
&= -\dbE|\wh\cY(t)|^2 +\lan\wh\cC(t)^{\top}M(t)\wh\cC(t)\dbE[\wh\cX(t)],\dbE[\wh\cX(t)]\ran,\\
{d\over dt}\lan N(t)\dbE[\wh\cX(t)],\dbE[\wh\cX(t)]\ran
&= -|\dbE[\wh\cX(t)]|^2 -\lan\wh\cC(t)^{\top}M(t)\wh\cC(t)\dbE[\wh\cX(t)],\dbE[\wh\cX(t)]\ran.
\end{align*}
Since the continuous functions $M(\cd)$ and $N(\cd)$ are $\t$-periodic and uniformly positive definite,
there exists a constant $\b>0$ such that
$$
M(t),N(t) \les \b^{-1} I_n, \q\forall t\ges0.
$$
Consequently,
\begin{align*}
& {d\over dt}\[\dbE\lan M(t)\wh\cY(t),\wh\cY(t)\ran+\lan N(t)\dbE[\wh\cX(t)],\dbE[\wh\cX(t)]\ran\] \\
&\q= -\[\dbE|\wh\cY(t)|^2 + |\dbE[\wh\cX(t)]|^2\] \\
&\q\les -\b\[\dbE\lan M(t)\wh\cY(t),\wh\cY(t)\ran+\lan N(t)\dbE[\wh\cX(t)],\dbE[\wh\cX(t)]\ran\].
\end{align*}
By Gronwall's inequality,
\begin{align*}
&\dbE\lan M(\t)\wh\cY(\t),\wh\cY(\t)\ran+\lan N(\t)\dbE[\wh\cX(\t)],\dbE[\wh\cX(\t)]\ran \\
&\q\les e^{-\b\t}\[\dbE\lan M(0)\wh\cY(0),\wh\cY(0)\ran+\lan N(0)\dbE[\wh\cX(0)],\dbE[\wh\cX(0)]\ran\].
\end{align*}
Since $M(\t)=M(0)>0$ and $N(\t)=N(0)>0$, the above implies that
\begin{align*}
d_{\scp M(0),N(0)}(\cL(\nu_1),\cL(\nu_2)) \les e^{-\b\t}d_{\scp M(0),N(0)}(\nu_1,\nu_2).
\end{align*}
That is, $\cL$ is a contraction mapping with respect to the distance $d_{\scp M(0),N(0)}(\cd\,,\cd)$.
Therefore, by the fixed-point theorem, the SDE \rf{PMFLQ:X1} admits a unique (in the distribution sense)
square integrable solution.
\end{proof}

Now we introduce the following cost functional:
\begin{align*}
J_\t(v(\cd))
&\deq \dbE\int_{0}^{\t}\bigg\{\lan R(t)v(t),v(t)\ran
      +2\Blan\begin{pmatrix}q_1(t) \\ r_1(t)\end{pmatrix}\!,
             \begin{pmatrix}\cX(t) \\ v(t)  \end{pmatrix}\Bran \\
&\hp{=\ } +\Blan\begin{pmatrix*}[l]\bQ_1(t) & \!\bS_1(t)^\top \\ \bS_1(t) & \!\bR_1(t)\end{pmatrix*}\!
	            \begin{pmatrix}\dbE[\cX(t)] \\ \dbE[v(t)]\end{pmatrix}\!,
	            \begin{pmatrix}\dbE[\cX(t)] \\ \dbE[v(t)]\end{pmatrix}\Bran\bigg\}dt.
\end{align*}
where
\begin{equation}\label{PMFLQ:Notation2}\begin{aligned}
&\bQ_1\deq \wh\cC^{\top}P\wh\cC+\hQ+\wh\Th^{\top}\hR\wh\Th+\wh\Th^{\top}\hS+\hS^{\top}\wh\Th,\q \bR_1\deq \bR+\hD^{\top}P\hD, \\
&\bS_1\deq-\hB^{\top}\vP,\q q_1\deq \wh\cC^{\top}P\si+\wh\Th^{\top}r+q,\q r_1\deq\hD^{\top}P\si+r.
\end{aligned}\end{equation}
We impose the following optimal control problem.

\ms

\no{\bf Problem (MFLQ)$_\t$.}
Find an $v_\t^*(\cd)\in\sU[0,\t]$ such that
$$
J_{\t}(v_\t^*(\cd))=\inf_{v(\cd)\in \cU[0,\t]}J_{\t}(v(\cd)).
$$

For the above Problem (MFLQ)$_\t$, we have the following result.

\begin{proposition}\label{PMFLQ:Propy}
Let {\rm\ref{A1}--\ref{A3}} hold. Then $(\cX_\t^*(\cd),v_\t^*(\cd))$ is an optimal pair of
Problem (MFLQ)$_\t$ if and only if the solution of
\begin{equation}\label{PMFLQ:y}\lt\{\begin{aligned}
&\dot{\cY}(t)+\wh\cA(t)^{\top}\cY(t)+\bQ_1\dbE[\cX_\t^*(t)]+\bS_1^{\top}\dbE[v_\t^*(t)]+q_1(t)=0, \q t\in[0,\t]\\
&\cY(0)=\cY(\t)
\end{aligned}\rt.\end{equation}
satisfies
\begin{equation}\label{PMFLQ:yCd}
\hB(t)^{\top}\cY(t)+\bS_1(t)\dbE[\cX_\t^*(t)]+R(t)v_\t^*(t)+\bR_1(t)\dbE[v_\t^*(t)]+r_1(t)=0,~\as~\ae~t\in[0,\t].
\end{equation}
\end{proposition}

\begin{proof}
$v_\t^*(\cd)$ is optimal if and only if
\begin{equation}\label{24-7-12}
J_{\t}(v_\t^*(\cd)+\e v(\cd))-J_{\t}(v_\t^*(\cd))\ges 0,\q \forall\e\in\dbR,~\forall v(\cd)\in\cU[0,\t].
\end{equation}
Let $\cX^{(v)}(\cd)$ be the solution of
\begin{equation*}\lt\{\begin{aligned}
&d\cX^{(v)}(t)=\[\cA(t)\cX^{(v)}(t)+\bar\cA(t)\dbE[\cX^{(v)}(t)]+B(t)v(t)+\bB(t)\dbE [v(t)]\]dt\\
&\hp{d\cX^{(v)}(t)=} +\[\cC(t)\cX^{(v)}(t)+\bar\cC(t)\dbE[\cX^{(v)}(t)]+D(t)v(t)+\bD(t)\dbE [v(t)]\]dW(t),\\
&\mu_{\scp\cX^{(v)}(0)} = \mu_{\scp\cX^{(v)}(\t)}.
\end{aligned}\rt.\end{equation*}
Applying a similar argument as in the proof of \autoref{prop:X1-wellposed} to the process
$\cZ(\cd)\deq\begin{pmatrix*}\cX_\t^*(\cd)\\ \cX^{(v)}(\cd)\end{pmatrix*}$,
we can choose an appropriate initial value $\cZ(0)$ such that $\cZ(0)$ and $\cZ(\t)$
are identically distributed.
Let $\h\cX(\cd)\deq\cX_\t^*(\cd)+\e\cX^{(v)}(\cd)$.
Then $\h\cX(\t)$ and $\h\cX(0)$ are identically distributed.
As a result, $\h\cX(\cd)$ is the solution to \eqref{PMFLQ:X1} with respect to the control
$\hv(\cd)\deq v_\t^*(\cd)+\e v(\cd)$.
Now, a direct computation shows that
\begin{align*}
&J_{\t}(v_\t^*(\cd)+\e v(\cd))-J_{\t}(v_\t^*(\cd))\\
&\q=\e^2\dbE\int_0^\t\bigg\{\lan Rv,v\ran
    +\Blan\begin{pmatrix*}[l]\bQ_1 & \!\bS_1^\top \\ \bS_1 & \!\bR_1\end{pmatrix*}\!
		  \begin{pmatrix}\dbE[\cX^{(v)}] \\ \dbE[v]\end{pmatrix}\!,
		  \begin{pmatrix}\dbE[\cX^{(v)}] \\ \dbE[v]\end{pmatrix}\Bran\bigg\}dt \\
&\hp{\q=\ } +\e\dbE\int_0^\t\Big\{\lan\bQ_1\dbE[\cX_\t^*]+\bS_1^{\top}\dbE[v_\t^*]+q_1,\cX^{(v)}\ran
            +\lan\bS_1\dbE[\cX_\t^*]+Rv_\t^*+\bR_1\dbE[v_\t^*]+r_1,v\ran\Big\}dt.
\end{align*}
On the other hand, we have by It\^{o}'s rule that
\begin{align*}
0&= \dbE\lan\cY(\t),\cX^{(v)}(\t)\ran-\dbE\lan\cY(0),\cX^{(v)}(0)\ran \\
&= \dbE\int_0^\t\Big\{-\lan\bQ_1\dbE[\cX_\t^*]+\bS_1^{\top}\dbE[v_\t^*]+q_1,\cX^{(v)}\ran+\lan\hB^\top\cY,v\ran\Big\}dt.
\end{align*}
Combining the above equalities, we obtain
\begin{align*}
&J_{\t}(v_\t^*(\cd)+\e v(\cd))-J_{\t}(v_\t^*(\cd))\\
&\q=\e^2\dbE\int_0^\t\bigg\{\lan Rv,v\ran
    +\Blan\begin{pmatrix*}[l]\bQ_1 & \!\bS_1^\top \\ \bS_1 & \!\bR_1\end{pmatrix*}\!
		  \begin{pmatrix}\dbE[\cX^{(v)}] \\ \dbE[v]\end{pmatrix}\!,
		  \begin{pmatrix}\dbE[\cX^{(v)}] \\ \dbE[v]\end{pmatrix}\Bran\bigg\}dt \\
&\hp{\q=\ } +\e\dbE\int_0^\t\Big\{\lan\hB^\top\cY+\bS_1\dbE[\cX_\t^*]+Rv_\t^*+\bR_1\dbE[v_\t^*]+r_1,v\ran\Big\}dt.
\end{align*}
Note that the integral after $\e^2$ is nonnegative.
Indeed, since $\hR_1(\cd)\deq R(\cd)+\bR_1(\cd)$ is uniformly positive definite and
\begin{align*}
&\dbE\bigg\{\lan Rv,v\ran + \Blan
    \begin{pmatrix*}[l]\bQ_1 & \!\bS_1^\top \\ \bS_1 & \!\bR_1\end{pmatrix*}\!
    \begin{pmatrix}\dbE[\cX^{v}] \\ \dbE[v]\end{pmatrix}\!,
    \begin{pmatrix}\dbE[\cX^{v}] \\ \dbE[v]\end{pmatrix}\Bran\bigg\} \\
&\q=\dbE\bigg\{\lan R\big(v-\dbE [v]\big),v-\dbE[v]\ran
    +\Blan\begin{pmatrix*}[l]\bQ_1 & \!\bS_1^\top \\ \bS_1 & \!\hR_1\end{pmatrix*}\!
          \begin{pmatrix}\dbE[\cX^{v}] \\ \dbE[v]\end{pmatrix}\!,
          \begin{pmatrix}\dbE[\cX^{v}] \\ \dbE[v]\end{pmatrix}\Bran\bigg\},
\end{align*}
it suffices to show that
$$
\begin{pmatrix*}[l]\bQ_1(t) & \!\bS_1(t)^\top \\\bS_1(t) & \!\hR_1(t)\end{pmatrix*}\ges 0, \q\ae~ t\ges 0,
$$
or equivalently,
$$
\bQ_1(t)-\bS_1(t)^{\top}\hR_1(t)^{-1}\bS_1(t)\ges 0, \q\ae~ t\ges 0.
$$
For this, let
$$\D(t)\deq -\hR_1(t)^{-1}[\hD(t)^{\top}P(t)\hC(t) +\hS(t)]. $$
Then
\begin{align*}
\bQ_1-\bS_1^{\top}\hR_1^{-1}\bS_1
&= \wh\cC^{\top}P\wh\cC+\hQ+\wh\Th^{\top}\hR\wh\Th+\wh\Th^{\top}\hS
   +\hS^{\top}\wh\Th- \vP\hB\hR_1^{-1}\hB^{\top}\vP \\
&= \hC^{\top}P\hC+\hQ +\wh\Th^{\top}\hR_1\wh\Th+ \wh\Th^{\top}(\hD^{\top}P\hC+\hS)
   +(\hC^{\top}P\hD+\hS^{\top})\wh\Th \\
&\hp{=\ } - (\hR_1\wh\Th+\hD^{\top}P\hC+\hS)^{\top}\hR_1^{-1}(\hR_1\wh\Th+\hD^{\top}P\hC+\hS) \\
&=\hC^{\top}P\hC+\hQ - (\hD^{\top}P\hC+\hS)^{\top}\hR_1^{-1}(\hD^{\top}P\hC+\hS), \\
&= \hC^{\top}P\hC + \hQ + (\hC^{\top}P\hD+\hS^{\top})\D
   +\D^{\top}(\hD^{\top}P\hC+\hS) + \D^{\top}\hR_1\D \\
&= \hC^{\top}P\hC + \hC^{\top}P\hD\D + \D^{\top}\hD^{\top}P\hC + \D^{\top}\hD^{\top}P\hD\D \\
&\hp{=\ } + \hQ + \D^{\top}\hS +\hS^{\top}\D + \D^{\top}\hR\D \\
&= (\hC+\hD\D)^{\top}P(\hC+\hD\D) + (\hR\D+\hS)^{\top}\hR^{-1}(\hR\D+\hS) \\
&\hp{=\ } + \hQ -\hS^{\top}\hR^{-1}\hS \ges 0.
\end{align*}
Thus, \rf{24-7-12} holds if and only if
$$
\dbE\int_0^\t\Big\{\lan\hB^\top\cY+\bS_1\dbE[\cX_\t^*]+Rv_\t^*+\bR_1\dbE[v_\t^*]+r_1,v\ran\Big\}dt=0,
\q\forall v(\cd)\in\cU[0,\t],
$$
or equivalently, if and only if \rf{PMFLQ:yCd} holds.
\end{proof}

The following result provides an explicit expression for the optimal control $v_\t^*(\cd)$.

\begin{proposition}
Let {\rm\ref{A1}--\ref{A3}} hold and $\hR_1(\cd)\deq R(\cd)+\bR_1(\cd)$.
The unique optimal control $v_\t^*(\cd)$ of problem (MFLQ)$_\t$ is given by
\begin{equation}\label{PMFLQ:u_optimal}
v_\t^*(t) = -\hR_1(t)^{-1}\[\hB(t)^\top\eta_\t(t)+r_1(t)\],
\end{equation}
where $\eta_\t(\cd)$ is the solution to
\begin{equation}\label{PMFLQ:eta}\lt\{\begin{aligned}
&\dot{\eta}_\t(t) + \wh\cA(t)^\top\eta_\t(t) + q_1(t) + \vP(t)b(t) =0,\\
&\eta_\t(0) = \eta_\t(\t).
\end{aligned}\rt.\end{equation}
\end{proposition}

\begin{proof}
Taking expectations on both sides of \rf{PMFLQ:yCd} and then subtracting
the resulting equation from \rf{PMFLQ:yCd}, we obtain
$$
R(t)\(v_\t^*(t)-\dbE[v_\t^*(t)]\)=0, \q\as~\ae~t\in[0,\t],
$$
which implies that $v_\t^*(t)=\dbE[v_\t^*(t)]~\ae~t\in[0,\t]$, since $R(\cd)$ is uniformly positive definite.
Thus, $v_\t^*(\cd)$ is deterministic.
Let $\cY(\cd)$ be the solution of \rf{PMFLQ:y}
and set $\eta_{\t}(t)\deq \cY(t)-\vP(t)\dbE[\cX_\t^*(t)]$. Then noting that
$$
\dot{\vP}(t)+\wh\cA(t)^{\top}\vP(t)+\vP(t)\wh\cA(t)+\bQ_1(t)=0,
$$
we have
\begin{align*}
\dot{\eta}_{\t}(t)
&= \dot{\cY}(t)-\dot{\vP}(t)\dbE[\cX_\t^*(t)]-\vP(t){d\over dt}\dbE[\cX_\t^*(t)]\\
&= -\[\wh\cA(t)^{\top}\cY(t)+\bQ_1(t)\dbE[\cX_\t^*(t)]+\bS_1(t)^{\top}v_\t^*(t)+q_1(t)\]-\dot{\vP}(t)\dbE[\cX_\t^*(t)] \\
&\hp{=\ } -\vP(t)\[\wh\cA(t)\dbE[\cX_\t^*(t)]+\hB(t)v_\t^*(t)+b(t)\] \\
&= -\wh\cA(t)^\top\eta_\t(t)-\[\bS_1(t)^\top+\vP(t)\hB(t)\]v_\t^*(t)-q_1(t) \\
&\hp{=\ } -\[\dot{\vP}(t)+\wh\cA(t)^{\top}\vP(t)+\vP(t)\wh\cA(t)+\bQ_1(t)\]\dbE[\cX_\t^*(t)]-\vP(t)b(t)\\
&= -\wh\cA(t)^\top\eta_\t(t)-q_1(t)-\vP(t)b(t).
\end{align*}
Therefore, $\eta_{\t}(\cd)$ satisfies \rf{PMFLQ:eta}.
In terms of $\eta_{\t}(\cd)$, \rf{PMFLQ:yCd} becomes
\begin{align*}
0&=\hB(t)^{\top}\eta_{\t}(t)+\[\hB(t)^{\top}\vP(t)+\bS_1(t)\]\dbE[\cX_\t^*(t)]+\hR_1(t)v_\t^*(t)+r_1(t)\\
 &=\hB(t)^{\top}\eta_{\t}(t)+\hR_1(t)v_\t^*(t)+r_1(t), \q\ae~t\in[0,\t],
\end{align*}
from which we obtain \rf{PMFLQ:u_optimal}.
Finally, by using a similar argument to \cite[Remark 6.4]{Sun-Yong2024:JDE}, 
\rf{PMFLQ:eta} admits a unique solution.
Thus, the optimal control $v_{\t}^*(\cd)$ is also unique.
\end{proof}

We now extend the optimal pair $(\cX_\t^*(\cd),v_\t^*(\cd))$ of Problem (MFLQ)$_\t$
to $[0,\i)$ periodically as follows.
For any $t\in[0,\i)$, there exist a unique integer $k$ such that $t\in [k\t,(k+1)\t)$.
For such a $t$, we define
\begin{equation}\label{PMFLQ:Def-etau}
\eta(t)\deq \eta_{\t}(t-k\t),\q v^*(t) \deq -\hR_1(t)^{-1}\[\hB(t)^{\top}\eta(t)+r_1(t)\].
\end{equation}
The above two functions are all $\t$-periodic, and $\eta(\cd)$ satisfies the following ODE:
$$
\dot{\eta}(t)+\wh\cA(t)^\top\eta(t)+q_1(t)+\vP(t)b(t)=0.
$$
Similar to \autoref{prop:X1-wellposed}, we can prove that the SDE
\begin{equation}\label{PMFLQ:Xstar}\begin{aligned}
d\cX(t) &= \{\cA(t)\cX(t)+\bar\cA(t)\dbE[\cX(t)]+\hB(t)v^*(t)+b(t)\}dt\\
	    &\hp{=\ } +\{\cC(t)\cX(t)+\bar\cC(t)\dbE[\cX(t)]+\hD(t)v^*(t)+\si(t)\}dW(t)
\end{aligned}\end{equation}
admits a unique (in the distribution sense) $\t$-periodic solution $\cX^*(\cd)$.

\section{The turnpike property}\label{Sec:TP}

In this section we establish the exponential turnpike property
for Problem (MFLQ)$_{\scT}$.
Let $(\bX_{\scT}(\cd),\bu_{\scT}(\cd))$ be the optimal pair of Problem (MFLQ)$_{\scT}$
for the initial state $x$.
Let $v^*(\cd)$ be defined in \rf{PMFLQ:Def-etau},
$\cX^*(\cd)$ the $\t$-periodic solution of \rf{PMFLQ:Xstar}, and
\begin{equation}\label{TP:TPlimit-u*}
u^*(t) \deq \Th(t)\{\cX^*(t)-\dbE[\cX^*(t)]\} + \wh\Th(t)\dbE[\cX^*(t)] + v^*(t).
\end{equation}

\smallskip

We have the following result, which establish the exponential turnpike
property of Problem (MFLQ)$_{\scT}$.

\begin{theorem}\label{TP:MainResult}
Let {\rm\ref{A1}--\ref{A3}} hold. There exist constants $K,\l>0$, independent of $T$, such that
\begin{equation}\label{TP:tpineq}
\dbE\[|\bX_{\scT}(t)-\cX^*(t)|^2+|\bu_{\scT}(t)-u^*(t)|^2\]
\les K\[e^{-\l t}+e^{-\l(T-t)}\], \q\forall t\in[0,T].
\end{equation}
\end{theorem}

In preparation for proving \autoref{TP:MainResult}, we first present the following result.
Recall the functions $\p_{\scT}(\cd)$, $\f_{\scT}(\cd)$, and $\eta(\cd)$ defined by
\rf{MFLQ:phi-rep}, \rf{MF:bareta-d}, and \rf{PMFLQ:Def-etau}, respectively.

\begin{proposition}\label{TP:tp-etavp}
Let {\rm\ref{A1}--\ref{A3}} hold. Then there exist constants $K,\l>0$ independent of $T$, such that
$$
|\eta(t)-\f_{\scT}(t)|+|\p_{\scT}(t)-v^*(t)|\les Ke^{-\l(T-t)},\q\forall t\in[0,T].
$$
\end{proposition}

\begin{proof}
Recall the notation \rf{PMFLQ:Notation1} and let
\begin{equation}\label{NOTE:AChAhC}\begin{aligned}
&\cA_{\scT}(t) \deq A(t)+B(t)\Th_{\scT}(t), \q
 \cC_{\scT}(t) \deq C(t)+D(t)\Th_{\scT}(t), \\
&\wh\cA_{\scT}(t) \deq \hA(t)+\hB(t)\wh\Th_{\scT}(t), \q
 \wh\cC_{\scT}(t) \deq \hC(t)+\hD(t)\wh\Th_{\scT}(t).
\end{aligned}\end{equation}
where $\Th_{\scT}(\cd)$ and $\wh\Th_{\scT}(\cd)$ are defined by
\rf{MFLQ:Th-rep} and \rf{MFLQ:whTH-rep}, respectively.
Then $h_{\scT}(\cd)\deq \eta(\cd)-\f_{\scT}(\cd)$ satisfies $h_{\scT}(T)=\eta(T)$, and
\begin{align*}
0&= \dot{h}_{\scT}(t)+\wh\cA_{\scT}(t)^{\top}h_{\scT}(t)+\[\wh\cA(t)-\wh\cA_{\scT}(t)\]^\top\eta(t)
    +\[P(t)\wh\cC(t)-P_{\scT}(t)\wh\cC_{\scT}(t)\]^{\top}\si(t) \\
&\hp{=\ } +\[\wh\Th(t)-\wh\Th_{\scT}(t)\]^{\top}r(t)+\[\vP(t)-\vP_{\scT}(t)\]b(t) \\
&= \dot{h}_{\scT}(t) + \wh\cA(t)^{\top}h_{\scT}(t) + l_{\scT}(t),
\end{align*}
where
$$
l_{\scT} \deq (\wh\cA_{\scT}-\wh\cA)^{\top}h_{\scT} + (\wh\cA-\wh\cA_{\scT})^{\top}\eta
+(P\cC-P_{\scT}\wh\cC_{\scT})^{\top}\si + (\wh\Th-\wh\Th_{\scT})^{\top}r(t) + (\vP-\vP_{\scT})b.
$$
Recall that $\wh\F_{\scp\Th}(\cd)$ is the solution to \eqref{SDE:hF_Th}.
Then
\begin{equation}\label{TP:CtVnFa}
h_{\scT}(t) = \[\wh\F_{\scp\Th}(T)\wh\F_{\scp\Th}(t)^{-1}\]^{\top}\eta(T)
+\int_{t}^{\scT}\[\wh\F_{\scp\Th}(s)\wh\F_{\scp\Th}(t)^{-1}\]^{\top}l_{\scT}(s)ds.
\end{equation}
By \autoref{prop:PT-P}, \autoref{thm:vP-vPT}, \autoref{CORO:Th-whTh} and \rf{INEQ:hF_Th},
there exist constants $K,\l>0$, independent of $T$, such that for any $0\les t\les s\les T$,
\begin{align*} &|\wh\cA(t)-\wh\cA_{\scT}(t)|+|P(t)\wh\cC(t)-P_{\scT}(t)\wh\cC_{\scT}(t)|+|\wh\Th(t)-\wh\Th_{\scT}(t)|+|\vP(t)-\vP_{\scT}(t)|\les Ke^{-\l(T-t)},\\
&|\wh\F_{\scp\Th}(s)\wh\F_{\scp\Th}(t)^{-1}|\les Ke^{-\l(s-t)}.
\end{align*}
Observing that $\eta(\cd)$ is bounded, \rf{TP:CtVnFa} implies that
\begin{align*}
|h_{\scT}(t)|&\les Ke^{-\l(T-t)}+K\int_{t}^{\scT}e^{-\l(s-t)}|l_{\scT}(s)|ds\\
&\les Ke^{-\l(T-t)}+K\int_{t}^{\scT}e^{-\l(s-t)}e^{-\l(T-s)}\[|h_{\scT}(s)|+1\]ds\\
&= Ke^{-\l(T-t)}+Ke^{-\l(T-t)}\int_{t}^{T}\[|h_{\scT}(s)|+1\]ds,
\end{align*}
where the constant $K$ is independent of $T$ and may vary from line to line. Now, set
$$
\beta_{\scT}(t)=|h_{\scT}(T-t)|e^{\l t}.
$$
Then the above can be written as
$$
\beta_{\scT}(t)\les K(1+t)+K\int_{0}^{t}e^{-\l s}\beta_{\scT}(s)ds.
$$
Applying Gronwall's inequality, we have
$$
\beta_{\scT}(t)\les K(1+t)\les \frac{2K}{\l}e^{\frac{\l}{2}t},
$$
for possibly different positive constant $K$. As a result, we obtain
$$
|h_{\scT}(t)|\les \frac{2K}{\l}e^{-\frac{\l}{2}(T-t)}, \q\forall t\in[0,T].
$$
This completes the proof.
\end{proof}

\begin{proof}[Proof of \autoref{TP:MainResult}]
Recall the notation introduced in \rf{PMFLQ:Notation1} and \rf{NOTE:AChAhC}, as well as the SDE given by \rf{PMFLQ:Xstar}.
It is easy to check that $\cX^*(t)$ satisfies
\begin{align*}
d\cX^*(t)&=\big\{\cA(t)\big(\cX^*(t)-\dbE[\cX^*(t)]\big)+\wh\cA(t)\dbE[\cX^*(t)]+\hB(t)v^*(t)+b(t)\big\}dt\\
&\hp{=\ }+\big\{\cC(t)\big(\cX^*(t)-\dbE[ \cX^*(t)]\big)+\wh\cC(t)\dbE[\cX^*(t)]+\hD(t)v^*(t)+\si(t)\big\}dW(t).
\end{align*}
Also, by substituting \rf{MFLQ:u-rep} into \rf{TP:state}, we obtain
\begin{equation*}\lt\{\begin{aligned} d\bX_{\scT}(t)&=\big\{\cA_{\scT}(t)\big(\bX_{\scT}(t)-\dbE[\bX_{\scT}(t)]\big)
+\wh\cA_{\scT}(t)\dbE[\bX_{\scT}(t)]+\hB(t)\p_{\scT}(t)+b(t)\big\}dt\\
&\hp{=\ }+\big\{\cC_{\scT}(t)\big(\bX_{\scT}(t)-\dbE[\bX_{\scT}(t)]\big)
+\wh\cC_{\scT}(t)\dbE[\bX_{\scT}(t)]+\hD(t)\p_{\scT}(t)+\si(t)\big\}dW(t),\\
\bX_{\scT}(0)&=x.
\end{aligned}\rt.\end{equation*}
Now, set
\begin{align*}
H_{\scT}(t)\deq \bX_{\scT}(t)-\cX^*(t),\q \cX^*_0(t)\deq \cX^*(t)-\dbE [\cX^*(t)], \q t\in[0,T].
\end{align*}
Then $H_{\scT}(0)=x-\cX^*(0)$ and
\begin{equation}\label{TP:H-formula}\begin{aligned}
dH_{\scT}(t)&=\Big\{\cA_{\scT}(t)\big(H_{\scT}(t)-\dbE[H_{\scT}(t)]\big)
+[\cA_{\scT}(t)-\cA(t)]\cX^*_0(t)+
\wh\cA_{\scT}(t)\dbE [H_{\scT}(t)]\\
&\hp{=\ }+ [\wh\cA_{\scT}(t)-\wh\cA(t)]\dbE[\cX^*(t)]+\hB(t)[\p_{\scT}(t)-v^*(t)]\Big\}dt\\
&\hp{=\ }+\Big\{\cC_{\scT}(t)\big(H_{\scT}(t)-\dbE[ H_{\scT}(t)]\big)+[\cC_{\scT}(t)-\cC(t)]\cX^*_0(t)+\wh\cC_{\scT}(t)\dbE [H_{\scT}(t)]\\
&\hp{=\ }+ [\wh\cC_{\scT}(t)-\wh\cC(t)]\dbE[\cX^*(t)]+\hD(t)[\p_{\scT}(t)-v^*(t)]\Big\}dW(t).
\end{aligned}\end{equation}
Taking expectation each side, we have $\dbE [H_{\scT}(0)]=x-\dbE[\cX^*(0)]$ and
\begin{equation}\label{TP:EH-formula}
d\dbE[H_{\scT}(t)]=\lt\{\wh\cA_{\scT}(t)\dbE[ H_{\scT}(t)]+ [\wh\cA_{\scT}(t)-\wh\cA(t)]\dbE[\cX^*(t)]+\hB(t)[\p_{\scT}(t)-v^*(t)]\rt\}dt.
\end{equation}
Subtracting \rf{TP:H-formula} from \rf{TP:EH-formula}, it follows that
\begin{equation}\label{TP:H-EH-formula}\begin{aligned}
d\big(H_{\scT}(t)-\dbE [H_{\scT}(t)]\big)&=\Big\{\cA_{\scT}(t)\big(H_{\scT}(t)-\dbE [H_{\scT}(t)]\big)+[\cA_{\scT}(t)-\cA(t)]\cX^*_0(t)\Big\}dt\\
&\hp{=\ }+\Big\{\cC_{\scT}(t)\big(H_{\scT}(t)-\dbE [H_{\scT}(t)]\big)+[\cC_{\scT}(t)-\cC(t)]\cX^*_0(t)+k(t)\Big\}dW(t),
\end{aligned}\end{equation}
where
\begin{align*}
	k(t)\deq \wh\cC_{\scT}(t)\dbE [H_{\scT}(t)]+[\wh\cC_{\scT}(t)-\wh\cC(t)]\dbE[\cX^*(t)]+\hD(t)[\p_{\scT}(t)-v^*(t)].
\end{align*}

In what follows, we shall estimate $|\dbE [H_{\scT}(t)]|$ and $|k(t)|$ first.
Note that \rf{TP:EH-formula} can be written as
\begin{align*}
\frac{d\dbE [H_{\scT}(t)]}{dt}&=\wh\cA(t)\dbE [H_{\scT}(t)]+[\wh\cA_{\scT}(t)-\wh\cA(t)]\dbE [H_{\scT}(t)]\\
&\hp{=\ }+ [\wh\cA_{\scT}(t)-\wh\cA(t)]\dbE[\cX^*(t)]+\hB(t)[\p_{\scT}(t)-v^*(t)].
\end{align*}
Recall that $\wh\F_{\scp\Th}(\cd)$ is the solution to \rf{SDE:hF_Th}.
Then by \autoref{CORO:Th-whTh}, \autoref{TP:tp-etavp} and \rf{INEQ:hF_Th},
for any $0\les s\les t\les T$, there exist constants $K,\l>0$, independent of $T$, such that
\begin{align*}
& |\wh\cA_{\scT}(t)-\wh\cA(t)|+ |\p_{\scT}(t)-v^*(t)|\les Ke^{-\l(T-t)}, \\
& |\wh\F_{\scp\Th}(t)\wh\F_{\scp\Th}(s)^{-1}|\les Ke^{-\l(t-s)}.\nn
\end{align*}
Further, since $\cX^*(\cd)$ is $\t$-periodic, then $|\dbE[\cX^*(\cd)]|$ is bounded.
Applying the variation of constants formula, it follows that
\begin{align*}
|\dbE[H_{\scT}(t)]|&\les |\wh\F_{\scp\Th}(t)|\cd|x-\dbE[\cX^*(0)]|+\int_{0}^{t}|\wh\F_{\scp\Th}(t)\wh\F_{\scp\Th}(s)^{-1}|\[|\wh\cA_{\scT}(s)-\wh\cA(s)|\cd|\dbE [H_{\scT}(s)]|\\
&\hp{\les\ } + |\wh\cA_{\scT}(s)-\wh\cA(s)|\cd|\dbE[\cX^*(s)]|+|\hB(s)|\cd|\p_{\scT}(s)-v^*(s)|\]ds\\
&\les Ke^{-\l t}+K\int_{0}^{t}e^{-\l(t-s)}e^{\l(T-s)}\[|\dbE [H_{\scT}(s)]|+1\]ds\\
&\les K\[e^{-\l t}+e^{-\l(T-t)}\]+K\int_{0}^{t}e^{-\l(T-s)}|\dbE [H_{\scT}(s)]|ds,
\end{align*}
for possibly different positive constants $K$ and $\l$.
For convenience, hereafter we shall use $K$ and $\l$ to denote two generic positive constants which do not depend on $T$ and may vary from line to line.
Now, by Gronwall's inequality, we obtain
\begin{equation}\label{TP:Re4}
|\dbE [H_{\scT}(t)]|\les K\[e^{-\l t}+e^{-\l(T-t)}\].
\end{equation}
As a result,
\begin{align}\label{TP:ES-kt}
|k(t)|&\les \[|\wh\cC_{\scT}(t)-\wh\cC(t)|+|\wh\cC(t)|\]|\dbE [H_{\scT}(t)]|+|\wh\cC_{\scT}(t)-\wh\cC(t)|\cd|\dbE[\cX^*(t)]| \nn\\
&\hp{\les\ }+|\hD(t)|\cd|\p_{\scT}(t)-v^*(t)|\les K\[e^{-\l t}+e^{-\l(T-t)}\].
\end{align}

\ms
	
Next, we turn to estimate $\dbE|H_{\scT}(t)-\dbE [H_{\scT}(t)]|^2$.
\autoref{prop:periodi-Ric} implies that $[\cA(\cd),\cC(\cd)]$ is mean-square exponentially stable.
Thus, according to \autoref{remark:Sun-Yong2024}, for a given positive constant $\d$,
there exists a unique $\t$-periodic, uniformly positive definite function
$M(\cd)\in C(0,\i;\dbS^{n}_{+})$ satisfying
$$
\dot{M}(t)+M(t)\cA(t)+\cA(t)^{\top}M(t)+\cC(t)^{\top}M(t)\cC(t)+2\d I_{n}=0, \q\forall t\ges 0.
$$
Further, there exist some constants $\alpha,\beta>0$, such that
\begin{equation}\label{TP:ES-M}
\a^{-1} I_{n}\les M(t)\les \beta^{-1} I_{n},\q \forall t\ges 0.
\end{equation}
Now, using the It\^{o}'s formula, \rf{TP:H-EH-formula} yields
\begin{align}\label{TP:PH}
&\dbE\big\lan M(t)\big(H_{\scT}(t)-\dbE [H_{\scT}(t)]\big),H_{\scT}(t)-\dbE [H_{\scT}(t)]\big\ran-\dbE\big\lan M(0)\cX^*_0(0),\cX_0^*(0)\big\ran \nn\\
&\q=\dbE\int_{0}^{t}\bigg\{\big\lan\dot{M}\big(H_{\scT}-\dbE[ H_{\scT}]\big),H_{\scT}-\dbE H_{\scT}\big\ran+\big\lan Mk,k\big\ran \nn\\
&\hp{\q=\ }+2\big\lan M\big(H_{\scT}-\dbE [H_{\scT}]\big),\cA_{\scT}\big(H_{\scT}-\dbE [H_{\scT}]\big)+[\cA_{\scT}-\cA]\cX^*_{0}\big\ran \nn\\
&\hp{\q=\ }+\big\lan M\big[\cC_{\scT}\big(H_{\scT}-\dbE[ H_{\scT}]\big)+(\cC_{\scT}-\cC)\cX^*_0\big],\cC_{\scT}\big(H_{\scT}
-\dbE [H_{\scT}]\big)+(\cC_{\scT}-\cC)\cX^*_0\big\ran\bigg\}ds \nn\\
&\q=\dbE\int_{0}^{t}\bigg\{\big\lan \big(\dot{M}+M\cA_{\scT}+\cA_{\scT}^{\top}M+\cC_{\scT}^{\top}M\cC_{\scT}\big)\big(H_{\scT}-\dbE [H_{\scT}]\big),H_{\scT}-\dbE [H_{\scT}]\big\ran \nn \\
&\hp{\q=\ }+2\big\lan H_{\scT}-\dbE [H_{\scT}],[M(\cA_{\scT}-\cA)+\cC_{\scT}^{\top}M(\cC_{\scT}-\cC)]\cX^*_0\ran\nn\\
&\hp{\q=\ }+\big\lan M(\cC_{\scT}-\cC)\cX^*_0,(\cC_{\scT}-\cC)\cX^*_0\big\ran+\big\lan Mk,k\big\ran\bigg\}ds.
\end{align}
\autoref{CORO:Th-whTh} implies that there exist constants $K,\l>0$, independent of $T$, such that
$$|\cA_{\scT}(t)-\cA(t)|+|\cC_{\scT}(t)-\cC(t)|\les Ke^{-\l(T-t)}.$$
Then it follows from \rf{TP:ES-M} that
\begin{align*}
&\dot{M}(t)+M(t)\cA_{\scT}(t)+\cA_{\scT}(t)^{\top}M(t)+\cC_{\scT}(t)^{\top}M(t)\cC_{\scT}(t)\\
&\q=\dot{M}(t)+M(t)\cA(t)+\cA(t)^{\top}M(t)+\cC(t)^{\top}M(t)\cC(t)+M(t)[\cA_{\scT}(t)-\cA(t)]\\
&\hp{\q=\ }+ [\cA_{\scT}(t)-\cA(t)]^{\top}M(t)+[\cC_{\scT}(t)-\cC(t)]^{\top}M(t)\cC(t)+\cC_{\scT}(t)^{\top}M(t)[\cC_{\scT}(t)-\cC(t)]\\
&\q\les \[-2\d+Ke^{-\l (T-t)}\]I_n,
\end{align*}
 which yields
\begin{align}\label{TP:Re1}
&\dbE\big\lan \big(\dot{M}(t)+M(t)\cA_{\scT}(t)+\cA_{\scT}(t)^{\top}M(t) \nn\\
&\q\q+\cC_{\scT}(t)^{\top}M(t)\cC_{\scT}(t)\big)\big(H_{\scT}(t)-\dbE
[H_{\scT}(t)]\big), H_{\scT}(t)-\dbE [H_{\scT}(t)]\big\ran\nn \\
&\q\les \[-2\d+Ke^{-\l(T-t)}\]\dbE|H_{\scT}(t)-\dbE [H_{\scT}(t)]|^2.
\end{align}
Note that $\dbE|\cX^*_0(\cd)|^2$ is bounded. Then by the Cauchy-Schwarz inequality, we obtain
\begin{align}\label{TP:Re2}
&2\dbE\big\lan H_{\scT}(t)-\dbE [H_{\scT}(t)],\big(M(t)[\cA_{\scT}(t)-\cA(t)]+\cC_{\scT}(t)^{\top}M(t)[\cC_{\scT}(t)-\cC(t)]\big)\cX^*_0(t)\big\ran\nn\\
&\q\les \d^{-1}\big|M(t)[\cA_{\scT}(t)-\cA(t)]+ \cC_{\scT}(t)^{\top}M(t)[\cC_{\scT}(t)-\cC(t)]\big|^2\dbE|\cX^*_0(t)|^2 \nn\\
&\hp{\q\les\ }+ \d\dbE|H_{\scT}(t)-\dbE [H_{\scT}(t)]|^2 \nn\\
&\q \les \d\dbE|H_{\scT}(t)-\dbE [H_{\scT}(t)]|^2+Ke^{-\l(T-t)}.
\end{align}
Moreover,
\begin{equation}\label{TP:Re3}
\dbE\big\lan M(t)[\cC_{\scT}(t)-\cC(t)]\cX^*_0(t),[\cC_{\scT}(t)-\cC(t)]\cX^*_0(t)\big\ran\les Ke^{-\l(T-t)},
\end{equation}
and
\begin{equation}\label{TP:Re5}
\big\lan M(t)k(t),k(t)\big\ran \les K\[e^{-\l t}+e^{-\l(T-t)}\].
\end{equation}
Now, set
$$G(t)\deq \dbE\big\lan M(t)\big(H_{\scT}(t)-\dbE [H_{\scT}(t)]\big),H_{\scT}(t)-\dbE [H_{\scT}(t)]\big\ran, \q t\in[0,T].$$
Combining \rf{TP:Re1}--\rf{TP:Re5} and noting \rf{TP:ES-M}, it follows from \rf{TP:PH} that
\begin{align*}
{dG(t) \over dt}&\les \[-\d+Ke^{-\l(T-t)}\]\dbE|H_{\scT}(t)-\dbE [H_{\scT}(t)]|^2+K\[e^{-\l t}+e^{-\l(T-t)}\] \\
&\les \[-\d\b+\a Ke^{-\l(T-t)}\]G(t)+K\[e^{-\l t}+e^{-\l(T-t)}\].
\end{align*}
Using the Gronwall's inequality and recalling \rf{TP:Re4}, we have
\begin{align*}
\dbE|\bX_{\scT}(t)-\cX^*(t)|^2&=|\dbE [H_{\scT}(t)]|^2+ \dbE|H_{\scT}(t)-\dbE [H_{\scT}(t)]|^2  \\
&\les |\dbE [H_{\scT}(t)]|^2+\a G(t) \les K \[e^{-\l t}+e^{-\l(T-t)}\],\q\forall t\in[0,T].
\end{align*}

Finally, since
\begin{align*}
\bu_{\scT}(t)-u^*(t)&=\Th_{\scT}(t)\[\bX_{\scT}(t)-\cX^*(t)\]+\[\wh\Th_{\scT}(t)-\Th_{\scT}(t)\]\dbE\[\bX_{\scT}(t)-\cX^*(t)\]\\
&\hp{=\ }+ \[\Th_{\scT}(t)-\Th(t)\]\cX^*_0(t)+\[\wh\Th_{\scT}(t)-\wh\Th(t)\]\dbE [\cX^*(t)]+\[\p_{\scT}(t)-v^*(t)\],
\end{align*}
we get the desired result \rf{TP:tpineq} immediately.
\end{proof}

In terms of the $L^2$-Wasserstein distance, the exponential turnpike property of Problem (MFLQ)$_{\scT}$ 
in \autoref{TP:MainResult} can be restated in the following manner. 

\begin{corollary}
Let {\rm\ref{A1}--\ref{A3}} hold. Let $\mu^*(t)$, $\nu^*(t)$, $\bar\mu_{\scT}(t)$ and $\bar\nu_{\scT}(t)$ denote the distributions of $\cX^*(t)$, $u^*(t)$, $\bX_{\scT}(t)$ and $\bu_{\scT}(t)$, respectively.
Then there exist constants $K,\l>0$, independent of $T$, such that
$$
d(\mu^*(t),\bar\mu_{\scT}(t)) + d(\nu^*(t),\bar\nu_{\scT}(t))
\les K\[e^{-\l t}+e^{-\l(T-t)}\],\q\forall t\in[0,T].
$$
\end{corollary}

\end{document}